\documentclass[11pt]{amsart}
\title{Open Gromov-Witten Theory and the Crepant Resolution Conjecture}
\author{Renzo Cavalieri and Dustin Ross}
\address{Renzo Cavalieri, Colorado State University, Department of Mathematics, Weber Building, Fort Collins, CO 80523-1874, USA}
\email{renzo@math.colostate.edu}
\address{Dustin Ross, Colorado State University, Department of Mathematics, Weber Building, Fort Collins, CO 80523-1874, USA}
\email{ross@math.colostate.edu}

\usepackage {color,graphicx,psfrag, verbatim,amssymb,amscd,enumerate,subfigure}
\usepackage{texdraw,xypic}

\newcommand{\proj}{\mathbb{P}}

\newcommand{\Z}{\mathbb{Z}}

\newcommand{\so}{\mathcal{O}}

\newcommand{\C}{\mathbb{C}}

\newtheorem*{mt1}{Theorem \ref{thm:ocrc}}
\newtheorem*{mt2}{Proposition \ref{prop:regglue}}
\newtheorem*{mt3}{Proposition \ref{prop:orbg}}
\newtheorem*{mt4}{Theorem \ref{thm:clcrc}}
\newtheorem{dummy}{}[section]
\newtheorem{lemma}[dummy]{Lemma}
\newtheorem{proposition}[dummy]{Proposition}
\newtheorem{theorem}[dummy]{Theorem}
\newtheorem{corollary}[dummy]{Corollary}

\theoremstyle{convention}
\newtheorem{convention}[dummy]{Orientation Convention}
\theoremstyle{definition}

\theoremstyle{remark}
\newtheorem{remark}[dummy]{Remark}

\usepackage{graphicx}

2
\begin{document}

\begin{abstract}
We compute open GW invariants for $\mathcal{K}_{\mathbb{P}^1}\oplus\mathcal{O}_{\mathbb{P}^1}$, open orbifold GW invariants for $[\C^3/\Z_2]$, formulate an open crepant resolution conjecture and verify it for this pair.  We show that open invariants can be glued together to deduce the Bryan-Graber closed crepant resolution conjecture for the orbifold $[\mathcal{O}_{\mathbb{P}^1}(-1)\oplus\mathcal{O}_{\mathbb{P}^1}(-1)/\Z_2]$ and its crepant resolution $\mathcal{K}_{\mathbb{P}^1\times\mathbb{P}^1}$.
\end{abstract}

\maketitle
\tableofcontents

\pagebreak

\section{Introduction}\label{sec:intro}

\subsection{Summary of Results}

We investigate two striking (conjectural) features of Gromov-Witten theory.
\begin{description}
\item[Crepant Transformation] the equivalence between GW theories of two targets related by a crepant birational transformation. In particular when the crepant transformation is the resolution of singularities of a Gorenstein orbifold, this is  referred to as the \textit{Crepant Resolution  Conjecture} (\textit{CRC}). 
\item[Gluing] the ability to recover GW invariants for a toric variety/orbifold from open invariants of open subspaces covering the target (\textit{Gluing}). 
\end{description} 
We give a complete and exhaustive description for the specific geometry in Figure \ref{square}. The global quotient 
$\mathfrak{X}=[\mathcal{O}_{\mathbb{P}^1}(-1)\oplus\mathcal{O}_{\mathbb{P}^1}(-1)/\Z_2]$ is a Hard Lefschetz orbifold having $Y=\mathcal{K}_{\mathbb{P}^1\times\mathbb{P}^1}$ (the total space of the canonical bundle of $\mathbb{P}^1\times\mathbb{P}^1$) as its crepant resolution. $\mathfrak{X}$ can be covered by two copies of $[\C^3/\Z_2]$, whose resolutions ($\cong \mathcal{K}_{\mathbb{P}^1}\oplus\mathcal{O}_{\mathbb{P}^1}$) cover $Y$.



\begin{figure}[b]
\xymatrix{
OGW_0(\mathcal{K}_{\mathbb{P}^1}\oplus \mathcal{O}_{\mathbb{P}^1}) \ar[rr]^{\footnotesize\begin{array}{c}Gluing:\\ Prop\ \ref{prop:regglue}\end{array}} \ar[dd]_{\footnotesize\begin{array}{c}OCRC:\\ Thm.\ \ref{thm:ocrc}\end{array}} & & GW_0(\mathcal{K}_{\mathbb{P}^1\times\mathbb{P}^1})\ar[dd]^{\footnotesize\begin{array}{c}CRC:\\ Thm.\ \ref{thm:clcrc}\end{array}}\\
	  	& & \\
	  	OGW_0[\C^3/\Z_2] \ar[uu]\ar[rr]^{\footnotesize\begin{array}{c}{Orb. Gluing:}\\ Prop\ \ref{prop:orbg}\end{array}}  & &\ar[uu] \hspace{0.5cm} GW_0[(\mathcal{O}_{\mathbb{P}^1}(-1)\oplus\mathcal{O}_{\mathbb{P}^1}(-1))/\Z_2]} 
\caption{Master diagram for the paper.}
\label{square}
\end{figure}

The four main results of this paper allow us to ``complete the square".

\begin{mt1}
We make and verify a Crepant Resolution Statement for the open invariants of $[\C^3/\Z_2]$ and $\mathcal{K}_{\mathbb{P}^1}\oplus\mathcal{O}_{\mathbb{P}^1}$.
\end{mt1}
This is the first occurence of a crepant resolution statement for open invariants. We compute the genus $0$ open potential for $[\C^3/\Z_2]$  (Prop. \ref{prop:orboppot}) using the methods of \cite{bc:ooinv}. In order to evaluate invariants for more than one boundary component we generalize Theorem 1 of  \cite{r:dd2}  to the case of two-part hyperelliptic Hodge integrals with an arbitrary number of descendant insertions (Thm. \ref{thm:tphi}). Although it may look like it, we point out that Theorem \ref{thm:tphi} is not an instance of the string equation in the orbifold case. The open potential for $ \mathcal{K}_{\mathbb{P}^1}\oplus\mathcal{O}_{\mathbb{P}^1}$ is computed (Prop. \ref{OGWres}) using the techniques of \cite{kl:oinv}. Some interesting classical combinatorics is required to package the potential in a manageable form.

\begin{mt2}
Closed invariants for an arbitrary toric CY threefold can be obtained by gluing open invariants.
\end{mt2}
We compare the contributions (to the restriction of the virtual fundamental class of the moduli space of stable maps to a given fixed locus) from the multiple covers of the fixed lines with the contributions of discs that glue to maps in that fixed locus. It is worth pointing out that our definition of the disc function is purely local (i.e. it does not depend on the global geometry of the threefold), hence these two contributions are not tautologically equal. It was recently pointed out to us that a similar check of the gluing occurred in  \cite[Appendix B]{df:gluing}.   

\begin{mt3}
Closed invariants for $\mathfrak{X}$ are recovered by gluing open invariants of $[\C^3/\Z_2]$.
\end{mt3}  

In the orbifold case we content ourselves with proving the gluing for the particular geometry that we are studying. Checking that orbifold invariants glue in general is currently under investigation by the second author.

\begin{mt4}
We verify the  CRC  ($\grave{a}$ la Bryan-Graber) for $\mathfrak{X}$ and $Y$.
\end{mt4}

We point out two interesting aspects of this result. First, while the CRC has been verified in many instances (\cite{bg:crc,bgp:crc,bg:hhicrc}), this is the first case in which the Bryan-Graber statement is checked for an orbifold which is not just a representation of a finite group. In a sense we are checking that the Bryan-Graber CRC has indeed a geometric content and is not just a group-theoretic feature of orbifold invariants. Second, we prove Theorem \ref{thm:clcrc} by showing that our open CRC is ``compatible with gluing", thus gathering some positive evidence that the CRC, in the toric case, may be addressed locally (see Section \ref{candm} for a discussion).

\subsection{Context and Motivation}
\label{candm}

The Atiyah-Bott localization theorem is effectively used in Gromov-Witten theory to reduce the computation of GW invariants for a toric target to a sum  of Hodge integrals over loci of fixed maps. Hodge integrals can be evaluated using Grothendieck Riemann Roch and Witten's conjecture, hence the slogan that localization turns toric GW theory into combinatorics. Alas, this slogan is more often than not a camouflaged admission of defeat for algebraic geometers, who are tipically unable to manage the combinatorial complexity and extract meaningful geometric information from GW invariants.   
From a physical point of view, open GW invariants (virtual counts of maps from bordered Riemann Surfaces) arise naturally from the propagation of open strings. Mathematically, they offer the opportunity to tackle the combinatorial complexity of GW invariants by making their study even more local. The strategy of the topological vertex (\cite{akv:ttv}) is to associate certain combinatorial gadgets to each fixed point of a toric variety, and give ``gluing rules" that reconstruct GW invariants. Philosophically (and physically), these gadgets should correspond to open invariants relative to branes intersecting the fixed lines containing the given vertex. In \cite{lllz:amtottv}, a limiting argument is used to motivate a mathematical theory of the topological vertex in terms of relative GW invariants. Katz and Liu (\cite{kl:oinv}) take a different approach towards open invariants: when the target admits an antiholomorphic involution $\sigma$, they define open invariants by picking the $\sigma$-invariant portion of the obstruction theory in ordinary GW theory.

In \cite{bc:ooinv}, Katz and Liu's approach is generalized in two different directions. First, it is noted that the construction can be made local: independently of the global geometry of the target, disc contributions to open invariants are computed by sitting a neighborhood of the fixed (affine) line where the disc is mapping inside a resolved conifold. This gives rise to a local theory that is very similar (and possibly identical) to the mathematical topological vertex of Li-Liu-Liu-Zhou. However it is now not straightforward that open invariants should glue correctly: this is the significance of Proposition \ref{prop:regglue}. The second generalization porters open invariants to the orbifold setting. The second author is currently working on a formulation of a general theory of the orbifold vertex, and of an extension of the gluing results. Our formulation of open invariants bypasses the technical problem that the foundations of relative stable maps to orbifolds have not yet been laid. One could also argue that the involution invariant approach is naturally tuned to the study of orbifold geometry (which is essentially ``locally $G$-invariant geometry").

Our opinion is that the worth of a local theory (especially if defined via localization) should be measured by its success in addressing global questions. One of the most intriguing conjectures in GW theory, the crepant resolution conjecture, predicts a relation between orbifold GW invariants of a Gorenstein orbifold and GW invariants of its crepant resolution (when it exists)\footnote{There are various incarnations of the CRC, of different level of  and generality. Here we only focus of the most concrete (and restrictive) version, which applies to our geometry. A nice survey of this rich story, containing the most general formulation, is \cite{cr:crc}}. A natural question is whether the CRC is compatible with gluing, and can therefore be addressed locally. In this paper we study this question for a simple and yet non-trivial geometry, and give a positive answer. 

Recently, Brini (\cite{b:rtam}) made a proposal, based on open mirror symmetry, on how to relate open invariants under crepant transformations. Verifying that his proposal agrees with Theorem \ref{thm:ocrc} is on our immediate agenda, as it would provide further evidence towards the validity of our program, but most importantly it would validate Brini's proposal as a conjectural formulation of a vertex CRC.

\subsection{Organization of the Paper}
In Section \ref{sec:prelims} we review open Gromov-Witten invariants and describe the computational methods for computing the open invariants of $\mathcal{K}_{\mathbb{P}^1}\oplus \mathcal{O}_{\mathbb{P}^1}$ and $[\C^3/\Z_2]$.  We finish the section by computing explicit formulas for certain hyper-elliptic Hodge integrals which show up in later computations.  Sections \ref{sec:openup} and \ref{sec:opendown} are the computational meat of the paper in which we compute all relevant open invariants.  In Section \ref{sec:openres} we show that the open invariants satisfy the open crepant resolution conjecture.  In section \ref{sec:glue}, we show that open invariants can be glued to obtain closed invariants.  Finally, in section \ref{sec:closedres} we show that the closed CRC for 
$[(\mathcal{O}_{\mathbb{P}^1}(-1)\oplus\mathcal{O}_{\mathbb{P}^1}(-1))/\Z_2]$ can be deduced from the open CRC.

\subsection{Acknowledgements}
Many thanks to  Dagan Karp for pointing our attention to this geometry as a natural first step in our program. To Melissa Liu for helpful comments about the paper and for pointing our attention to existing literature we were unaware of. To Andrea Brini for constant communication on his parallel work.  We are also grateful to Vincent Bouchard, Y.P. Lee, Sara Pasquetti, Yongbin Ruan, Hsian-Hua Tseng for many interesting discussions related to this project. Finally we would like to acknowledge the AIM workshop \textit{Recursion structures in topological string theory and enumerative geometry}, where the idea of a crepant resolution conjecture for open invariants was discussed.

\section{Preliminaries}\label{sec:prelims}

\subsection{Open Invariants}
In \cite{kl:oinv}, Katz and Liu propose a theory for computing open Gromov-Witten invariants, a generalization of ordinary Gromov-Witten theory computing virtual counts of maps from surfaces with boundary satisfying certain boundary conditions.   Consider a Calabi-Yau threefold $X$ and a special Lagrangian submanifold $L$.  Fix integers $g$ and $h$ and a relative homology class $\beta\in H_2(X,L;\Z)$ with $\partial\beta=\sum\gamma_i\in H_1(L,\Z)$.  Then the open Gromov-Witten invariant $N_{\beta;\gamma_1,...,\gamma_h}^{g,h}$
is a virtual count of maps $f:(\Sigma,\partial\Sigma)\rightarrow (X,L)$ satisfying
\begin{itemize}
\item $(\Sigma,\partial\Sigma)$ is a Riemann surface with genus $g$ and $h$ boundary components,
\item $f_*[\Sigma]=\beta$, and
\item $f_*[\partial\Sigma]=\sum\gamma_i$.
\end{itemize}

In order to compute open invariants, Katz and Liu propose an obstruction theory for the moduli space of open stable maps $\overline{\mathcal{M}}_{g,h}(X,L|\beta;\gamma_i)$ (\cite[section 4.2]{kl:oinv}).  Under the assumption that the moduli space can be equipped with a well-behaved torus action, they give an explicit formula for how the corresponding virtual cycle restricts to the fixed locus of the torus action.  A particularly interesting aspect of this theory is that the virtual cycle \textit{does} depend on the torus action.  In other words, different torus actions lead to different invariants.  This reflects the framing dependence of open invariants discussed in \cite{akv:dinst}.

The computational key to the Katz and Liu setup is the assumption that $L$ is the fixed locus of an anti-holomorphic involution.  A map from a bordered Riemann surface mapping boundary into $L$ can then be doubled to a map from a closed Riemann surface (\cite[section 3.3]{kl:oinv}).  Open Gromov-Witten invariants are defined/computed from the involution invariant contributions to the ordinary Gromov-Witten invariants corresponding to the doubled maps.

Katz and Liu then specialize to compute disk invariants of 
$\mathcal{O}_{\mathbb{P}^1}(-1)\oplus \mathcal{O}_{\mathbb{P}^1}(-1)$, where $L$ is the fixed locus of the anti-holomorphic involution $(z,u,v)\rightarrow (1/\bar{z},\bar{v}\bar{z},\bar{u}\bar{z})$.  Key to the computations are the Riemann-Hilbert bundles $L(2d)$ and $N(d)$ over $(D^2,S^1)$ defined in \cite[Examples 3.4.3 and 3.4.4]{kl:oinv}.  The sections of the Riemann-Hilbert bundles are identified  torus-equivariantly to the involution invariant sections of $H^0(\mathbb{P}^1,\mathcal{O}(2d))$ and $H^1(\mathbb{P}^1,\mathcal{O}(-d)\oplus\mathcal{O}(-d))$, respectively.  

Our open invariant computations stem from making Katz and Liu's construction ``local", as we now explain.
We represent a toric Calabi-Yau 3-fold  via its \textit{web-diagram}, a planar trivalent graph where edges correspond to torus invariant lines and vertices to torus invariant points.  Equipping the space with a $\C^*$ action and lifting the action to the moduli space of open stable maps, the fixed loci consists of maps decomposing as
\begin{itemize}
\item compact components of the source curve contracting to the vertices,
\item multiple covers of the fixed lines of the 3-fold, fully ramified over fixed points, and
\item disks mapping with appropriate winding to edges equipped with a lagrangian. 
\end{itemize}  
The contribution from the first two items can be computed using standard Atyiah-Bott localization  whereas the contribution from each disk is computed by applying the Katz-Liu setup to a formal neighborhood of the fixed point where the vertex of the disk is mapped.

\subsection{Orientation Convention}

A subtlety arises in the computations.  Although the sections of $H^0(L(2d))$ are \textit{naturally} isomorphic to the sections of $H^0(\so_{\proj^1}(2d))$, there is not a natural choice of isomorphism between the sections of $H^1(N(d))$ and the sections of $H^1(\so_{\proj^1}(-d)\oplus\so_{\proj^1}(-d))$.  Rather, the latter correspondence depends on a choice of \textit{orientation} for the sections (see \cite[Section 5.2]{kl:oinv}): a $\sigma$ invariant section of  $H^1(\so_{\proj^1}(-d)\oplus\so_{\proj^1}(-d))$ in local coordinates at $0$ has the form
\begin{equation*}
s=\left(\sum_{i=1}^{d-1}\frac{{a_i}}{z^{i}},\sum_{i=1}^{d-1}\frac{\overline{a_i}}{z^{d-i}}\right)=\left(\sum_{j=1}^{d-1}\frac{\overline{b_j}}{z^{d-j}},\sum_{j=1}^{d-1}\frac{{b_j}}{z^{j}}\right).
\end{equation*}
The space of involution invariant sections is identified (torus-equivariantly) with a complex vector space by the first projection if using the coordinates $a_i$, or by second projection if using the coordinates $b_j$.  This choice results in different open invariants: in the first case the weights of the sections are the $\C^\ast$-weights of the sections of the left hand side $\so_{\proj^1}(-d)$, in the second case of the right hand side $\so_{\proj^1}(-d)$. Ultimately the choice of orientation yields a global factor of $(-1)^{d+1}$ where $d$ is the winding of the disk.


In order to track the choice of orientation, we make the following convention. 

\begin{convention}
Throughout the paper, we add an arrow to each edge intersecting a Lagrangian. The corresponding disk contributions are computed identifying the involution invariant sections of $H^1(\so_{\proj^1}(-d)\oplus\so_{\proj^1}(-d))$ via projection to the sections of the bundle to the \textit{left} of the arrow.
\end{convention}


\begin{figure}
\includegraphics[height=2cm]{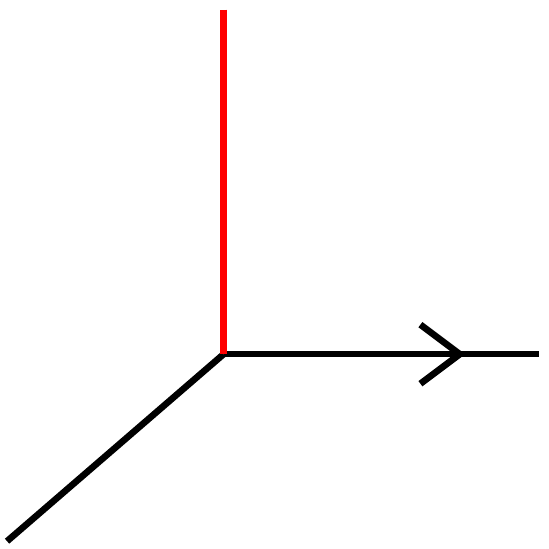}
\caption{$\C^3$ with one oriented half-edge denotes both that we are computing open invariants with disks lying along the horizontal edge and distinguishing that the identifications of $\sigma$ invariant sections is obtained by projecting onto the bundle corresponding to the vertical edge.}
\label{vert}
\end{figure}

In \cite{bc:ooinv}, the methods of Katz and Liu are extended to the orbifolds $[\C^3/\Z_n]$.  Analagous to computing closed \textit{orbifold} Gromov-Witten invariants, the open \textit{orbifold} Gromov-Witten invariants of $[\C^3/\Z_n]$ are defined/computed by considering only the contributions to the open invariants which descend to the quotient.  In both \cite{kl:oinv} and \cite{bc:ooinv}, the open invariants defined via the A-model are verified against B-model predictions.




\begin{figure}
  \begin{center}
    \subfigure[$\mathcal{K}_{\mathbb{P}^1}\oplus\mathcal{O}_{\mathbb{P}^1}$]{\label{fig:edge-a}\includegraphics[scale=0.6]{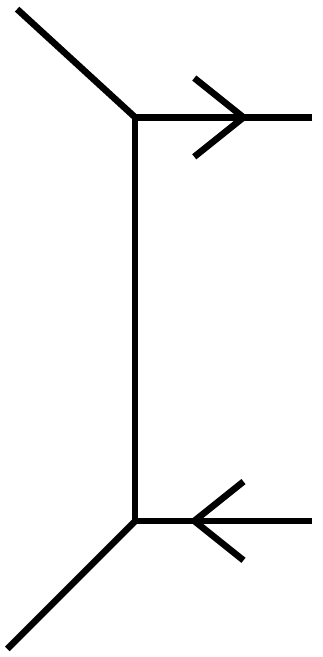}}\hspace{2cm}
    \subfigure[$\mathcal{K}_{\mathbb{P}^1\times\mathbb{P}^1}$]{\label{fig:edge-b}\includegraphics[scale=0.6]{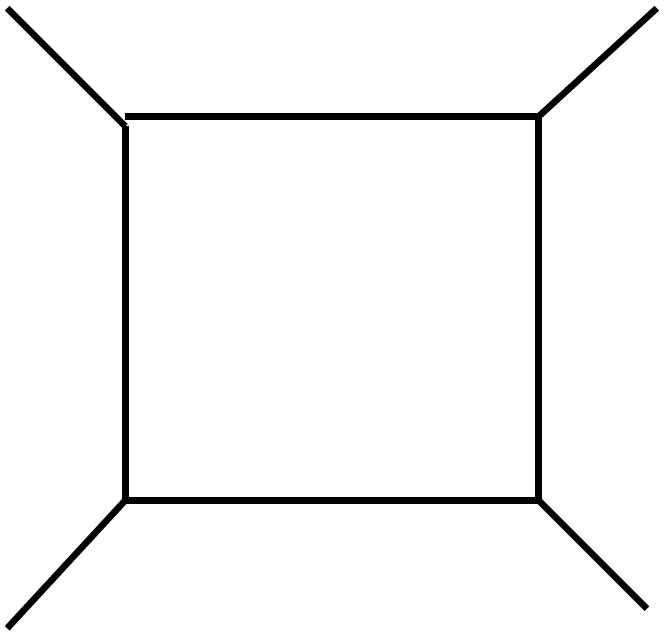}}\\
    \subfigure[{[$\C^3/\Z_2$]}]{\label{fig:edge-c}\includegraphics[scale=0.6]{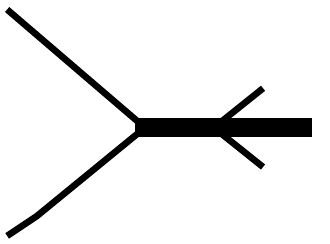}}\hspace{2cm}
    \subfigure[{[$(\mathcal{O}_{\mathbb{P}^1}(-1)\oplus\mathcal{O}_{\mathbb{P}^1}(-1))/\Z_2$]}]{\label{fig:edge-d}\includegraphics[scale=0.65]{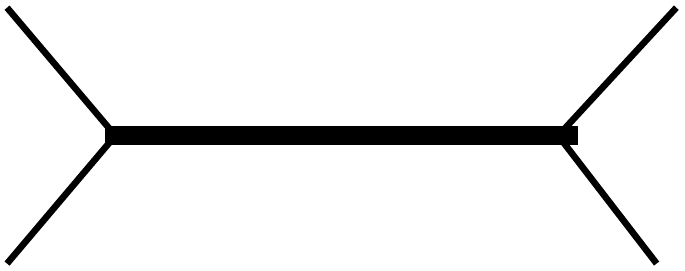}}
  \end{center}
  \caption{$\mathcal{B}\Z_2$ gerbes are denoted in bold and orientations have been chosen using the convention above.}
  \label{spaces}
\end{figure}



\subsection{Hyperelliptic Hodge Integrals}\label{sec:hodge}

In this section we prove a closed formula for a generating function which packages the hyperelliptic Hodge integrals of the form
\begin{equation}
L(g,i,\overline m):=\int_{\overline{\mathcal{M}}_{0;2g+2,0}(\mathcal{B}\Z_2)}\lambda_g\lambda_{g-i}(\overline\psi)^{\overline m}
\end{equation}
where $\overline m$ is a multi-index $(m_1,...,m_l)$, $|\overline m|:= m_1+...+m_l=i-1$, and
\begin{equation*}
\overline\psi^{\overline m}:= \psi_1^{m_1}\cdot...\cdot\psi_l^{m_l}.
\end{equation*}

\begin{remark}
Recall that $\overline{\mathcal{M}}_{0;2g+2,0}(\mathcal{B}\Z_2)$ is the moduli space of maps from genus zero curves into $\mathcal{B}\Z_2$ with $(2g+2)$ twisted marked points.  Each such map corresponds to a (possibly nodal) genus $g$ double cover of the source curve ramified over the marked points.  We have two natural forgetful  maps:
\begin{equation}
\xymatrix{
\overline{\mathcal{M}}_{0;2g+2,0}(\mathcal{B}\Z_2) \ar[r]^{\hspace{1cm}F} \ar[d]_\pi &\overline{\mathcal{M}}_{g}\\
\overline{\mathcal{M}}_{0;2g+2} }
\end{equation}
by sending a map to the corresponding double cover of its source curve.  The lambda classes on $\overline{\mathcal{M}}_{0;2g+2,0}(\mathcal{B}\Z_2)$ are defined to be
\begin{equation*}
\lambda_i:=c_i(F^*\mathbb{E})
\end{equation*}
where $\mathbb{E}$ is the Hodge bundle on $\overline{\mathcal{M}}_{g}$.  The psi classes are defined
via pull-back from $\overline{\mathcal{M}}_{0;2g+2}$.
\end{remark}

For a fixed $i$ and $\overline m$ with $|\overline m|=i-1$, define the generating function
\begin{equation}
\mathcal{L}_{i,\overline m}(x):=\sum_g L(g,i,\overline m)\frac{x^{2g}}{(2g)!}.
\end{equation}
We know from the $\lambda_g\lambda_{g-1}$ computation \cite{fp:lsahiittr,bp:tlgwtoc,bct:gg-1} that 
\begin{equation}
\label{fapala}
\mathcal{L}_{1,\emptyset}=\log\sec\left(\frac{x}{2}\right)
\end{equation}
and we also know from \cite{r:dd2} that 
\begin{equation}\label{base}
\mathcal{L}_{i,(i-1)}=\frac{2^{i-1}}{i!}\mathcal{L}_{1,\emptyset}^i
\end{equation}

The following theorem generalizes (\ref{base}).

\begin{theorem}
\label{thm:tphi}
\begin{equation}
\label{tphi}
\mathcal{L}_{i,\overline m}={m_1+...+m_l \choose m_1,...,m_l}\frac{2^{i-1}}{i!}\mathcal{L}_{1,\emptyset}^i.
\end{equation}
\end{theorem}

\begin{remark}
This formula appeared independently in Danny Gillam's PhD dissertation.  He computationally verified the result for $l\leq 4$.
\end{remark}

\begin{proof}

We use induction on the multi-index $\overline{m}$.  Given a multi-index $\overline{m}=(m_1,...,m_k)$ with $|\overline{m}|=j-1$, we know the result is true if either $j=1$ or $k=1$.    Suppose the lemma holds when 
\begin{enumerate}
\item $j<i$ and
\item $j=i, k \leq l$.  
\end{enumerate}

Under these assumptions, we show (\ref{tphi}) holds when $j=i$ and $k=l+1$.

\noindent{\bf Notation.}  Write $\overline m=(m_1,...,m_l,m_{l+1})$ and set $\overline m'=(m_1,...,m_{l-1},m_l')$ where $m_l':= m_l+m_{l+1}$.    For a subset $A\subseteq \{1,...,l+1\}$, we  write $\overline{m}(A)$ for the multi-index which is equal to $\overline{m}$ in the entries indexed by numbers in $A$ and equal to $0$ in the other entries.  $A^c$  denotes the complement of $A$.  $\overline{m} [k]$ denotes the multi-index $\overline{m}$ with the first entry replaced by $k$.

We prove the recursion by evaluating via localization auxiliary integrals on $\overline{\mathcal{M}}_{0;2g+2,0}(\mathbb{P}^1\times\mathcal{B}\Z_2,1)$.  This moduli space parametrizes double covers of the source curve with a special rational component picked out.  By postcomposing the usual evaluation maps with projection onto the first factor, we have evaluation maps to $\mathbb{P}^1$ which we denote by $e_i$.  The auxiliary integrals are:
\begin{description}
\item[$A1$] \begin{equation*}\int\lambda_g\lambda_{g-i}\overline{\psi}^{\overline{m}(\{1\}^c)}e_{l}^*(0)e_{l+1}^*(0)e_{2g+2}^*(\infty)\end{equation*}
\item[$A2$]  \begin{equation*}\int\lambda_g\lambda_{g-i}\overline{\psi}^{\overline{m'}(\{1\}^c)}e_{l}^*(0)e_{l+1}^*(0)e_{2g+2}^*(\infty)\end{equation*}
\end{description}

\begin{remark}

\begin{enumerate}
\item In each integrand, we  do not include the $\psi_1$ part of the Hodge integral.  The $\psi_1$ classes in the result make an appearance through node smoothing.  The other $\psi$ classes correspond to the marked points with the matching index.
\item We have abused notation in order to make the expression legible. By $\lambda_i$ we intend $c_{g-i}^{eq.}(R^1\pi_*f^*\mathcal{O})$ where the trivial bundle is linearized with $0$ weights: the lambda classes are how these classes restrict to the fixed loci. By $e_i^\ast(0)$ (resp. $e_i^\ast(\infty)$) we denote $c_1^{eq.}(e_{l}^*\mathcal{O}(1))$ linearized with weight $1$ over $0$ and weight $0$ over $\infty$ (resp. $0$ over $0$ and $-1$ over $\infty$). These classes essentially localize to require the corresponding mark point to map over $0$ (resp. $\infty$).
\item The difference in the two auxiliary integrals is that we have ``spread" the $\psi$ classes on the two points fixed over $0$ in two different ways. 
\item Both integrals vanish by dimensional reasons.  In both integrals the degree of the class we integrate is $m_2+...+m_{l+1}+3+2g-i$ and this is  strictly less than $2g+2$ (because $m_1+...+m_{l+1}=i-1$ and $m_1>0$). 
\item Localizing $A1$ yields relation (\ref{R1}) among Hodge integrals where all terms are already known by induction. Localizing $A2$ we get a relation (\ref{R2}) computing one unknown Hodge integral in terms of inductively known ones. Noticing that (\ref{R1}) and (\ref{R2}) are proportional to each other allows one to determine the desired integral.
\end{enumerate}
\end{remark}




Analyzing the obstruction theory via the normalization sequence of the source curve, one sees that the maps in the contributing fixed loci satisfy the following properties (\cite{r:adm} for more details):
\begin{itemize}
\item The preimages of $0$ and $\infty$ in the corresponding double cover must be connected.
\item One distinguished projective line in the source curve maps to the main component of the target with degree $1$.  The corresponding double cover has a rational component over the distinquished projective line.
\item The $l$th and $(l+1)$th marked points must map to $0$ while the $(2g+2)$th marked point must map to $\infty$.  
\end{itemize}
The contributing fixed loci are:
\begin{description}
\item [$F_g$] All marked points except for the $(2g+2)$th map to $0$.  The corresponding double cover contracts a genus $g$ component over $0$ and does not have a positive dimensional irreducible component over $\infty$.  This locus is isomorphic to $\overline{\mathcal{M}}_{0;2g+2,0}(\mathcal{B}\Z_2)$.
\item [$F_{g_1,g_2}$] $2g_1+1$ marked points map to $0$ and $2g_2+1$ marked points map to $\infty$ (this includes the points that are already forced to map to $0$ and $\infty$).  The corresponding double cover contracts a genus $g_1$ component over $0$ and a genus $g_2$ component over $\infty$.  This locus is isomorphic to $\overline{\mathcal{M}}_{0;2g_1+2,0}(\mathcal{B}\Z_2)\times\overline{\mathcal{M}}_{0;2g_2+2,0}(\mathcal{B}\Z_2)$.
\end{description}
The mirror analog of $F_g$ is not in the fixed locus because we are requiring that at least $2$ marked points map to $0$.

The first integral evaluates on the two types of fixed loci to:
\begin{description}
\item[($F_g$)]
\begin{equation*}
\frac{(-1)^i}{t^{m_1}}\int_{\overline{\mathcal{M}}_{0;2g+2,0}(\mathcal{B}\Z_2)}\lambda_g\lambda_{g-i}\overline{\psi}^{\overline{m}}=\frac{(-1)^i}{t^{m_1}}L(g,i,\overline{m})
\end{equation*}
\item[($F_{g_1,g_2}$)]
\begin{align*}
&\frac{2(-1)^i}{t^{m_1}}\sum_{k=1}^{i-1}\sum_{A\subseteq\{2,...,l-1\}}{2g+1-l \choose 2g_1+1-|A|}(-1)^{k-|\overline{m}(A^c)|-1}\\
&\hspace{1.5cm}\cdot\int_{\overline{\mathcal{M}}_{0;2g_1+2,0}(\mathcal{B}\Z_2)}\lambda_{g_1}\lambda_{g_1-i+k}\psi_1^{i-k-|\overline{m}(A) |-1}\overline{\psi}^{\overline{m}(A)}\psi_l^{m_l}\psi_{l+1}^{m_{l+1}}\\
&\hspace{3cm}\cdot\int_{\overline{\mathcal{M}}_{0;2g_2+2,0}(\mathcal{B}\Z_2)}\lambda_{g_2}\lambda_{g_2-k}\psi_1^{k-|\overline{m}(A^c)|-1}\overline{\psi}^{\overline{m}(A^c)}\\
\end{align*}
\end{description}

where we only sum over subsets $A$ which keep the powers of $\psi$ classes nonnegative.  The subset $A$ determines which $\psi$ classes map to $0$ and the binomial coefficient corresponds to the number of ways to distribute the marked points with no corresponding $\psi$ class in the integral.

Now write $\overline{n}_{A,k}$ for the multi-index $\overline{m}(A^c)[k-|\overline{m}(A^c) |-1]$.  The vanishing of the integral and the above computations yield the following relation:
\begin{align}
\label{R1a}
L(g,i,\overline{m})&=2\sum_{g_1}\sum_{k=1}^{i-1}\sum_{A\subseteq\{2,...,l-1\}}{2g+1-l \choose 2g_1+1-|A|}(-1)^{k-|\overline{m}(A^c) |} \nonumber\\
&\hspace{1.5cm}\cdot L(g_1,i-k,\overline{m}-\overline{n}_{A,k})\cdot L(g_2,k,\overline{n}_{A,k})
\end{align}
Evaluating the auxiliary integral for all genera and packaging (\ref{R1a}) in generating function form:
\begin{align}
\label{R1}
&\frac{d^{l-1}}{dx^{l-1}}\mathcal{L}_{i,\overline{m}}= \nonumber\\
&2\sum_{k=1}^{i-1}\sum_{A\subseteq\{2,...,l-1\}}\hspace{-.5cm}
(-1)^{k-|\overline{m}(A^c) |}
\left(\frac{d^{l-1-|A|}}{dx^{l-1-|A|}}\mathcal{L}_{i-k,\overline{m}-\overline{n}_{A,k}}\right)
\left(\frac{d^{|A|}}{dx^{|A|}}\mathcal{L}_{k,\overline{n}_{A,k}}\right)
\end{align}
The second integral leads to a very similar relation:
\begin{align}
\label{R2}
&\frac{d^{l-1}}{dx^{l-1}}\mathcal{L}(i,\overline{m}')=\nonumber\\
&2\sum_{k=1}^{i-1}\sum_{A\subseteq\{2,...,l-1\}}\hspace{-.5cm}
(-1)^{k-|\overline{m}'(A^c) |}
\left(\frac{d^{l-1-|A|}}{dx^{l-1-|A|}}\mathcal{L}_{i-k,\overline{m}'-\overline{n}_{A,k}'}\right)
\left(\frac{d^{|A|}}{dx^{|A|}}\mathcal{L}_{k,\overline{n}'_{A,k}}\right)
\end{align}

By definition, $\overline{n}_{A,k}=\overline{n}'_{A,k}$, so
\begin{equation}
\frac{d^{|A|}}{dx^{|A|}}\mathcal{L}_{k,\overline{n}_{A,k}}=\frac{d^{|A|}}{dx^{|A|}}\mathcal{L}_{k,\overline{n}'_{A,k}}
\end{equation}
Also, the induction hypothesis implies (because $k\geq 1$) that
\begin{equation}
\frac{d^{l-1-|A|}}{dx^{l-1-|A|}}\mathcal{L}_{i-k,\overline{m}-\overline{n}_{A,k}}=\frac{(m_l+m_{l+1})!}{m_l!m_{l+1}!}\frac{d^{l-1-|A|}}{dx^{l-1-|A|}}\mathcal{L}_{i-k,\overline{m}'-\overline{n}_{A,k}'}
\end{equation}

Therefore  the left hand sides of (\ref{R1}) and (\ref{R2}) are term by term proportional and we can conclude,
\begin{equation}
\frac{d^{l-1}}{dx^{l-1}}\mathcal{L}_{i,\overline{m}}=\frac{(m_l+m_{l+1})!}{m_l!m_{l+1}!}\frac{d^{l-1}}{dx^{l-1}}\mathcal{L}_{i,\overline{m}'}.
\end{equation}

Now recall that $l(\overline{m})=l+1$, so in order for $\int\lambda_g\lambda_{g-i}\overline{\psi}^{\overline{m}}$ to be defined, we must have at least $l+1$ marked points in our moduli space.  Thus, in order to get a nontrivial integral, we must have $2g+2\geq l+1$.  All coefficients of monomials of smaller degree  than $x^{l-1}$ in both generating functions vanish and we can conclude that
\begin{align}
\mathcal{L}_{i,\overline{m}}&=\frac{(m_l+m_{l+1})!}{m_l!m_{l+1}!}\mathcal{L}_{i,\overline{m}'} \nonumber \\
&=\frac{(m_l+m_{l+1})!}{m_l!m_{l+1}!}{m_1+...+m_l' \choose m_1,...,m_l'}\frac{2^{i-1}}{i!}\mathcal{L}_{1,\emptyset}^i \nonumber \\
&={m_1+...+m_{l+1} \choose m_1,...,m_{l+1}}\frac{2^{i-1}}{i!}\mathcal{L}_{1,\emptyset}^i
\end{align}
where we use the induction hypothesis again on the second equality.
\end{proof}

All $\mathcal{L}_{i,\overline m}$ can be further packaged in one jumbo generating function (with infinitely many symmetric variables $q_i$ keeping track of all possible descendant insertions): 
\begin{equation}
\mathcal{L}(x,\overline{q}):=\sum_{i,\overline m}\mathcal{L}_{i,\overline m}\overline{q}^{\overline{m}}
\end{equation}

\begin{corollary}\label{gen}
\begin{equation}
\mathcal{L}=\frac{1}{\left(2\sum{q_i}\right)}\exp\left({\left(2\sum{q_i}\right)\mathcal{L}_{1,\emptyset}}\right)=\frac{1}{2\sum{q_i}}\sec^{2\sum{q_i}}{\left(\frac{x}{2}\right)}
\end{equation}
\end{corollary}

\begin{proof}  
The first equality follows immediately from theorem (\ref{thm:tphi}). The second is obtained by plugging (\ref{fapala}) for $\mathcal{L}_{1,\emptyset}$.

\end{proof}

\section{Open Gromov-Witten Invariants of $\mathcal{K}_{\mathbb{P}^1}\oplus\mathcal{O}_{\mathbb{P}^1}$}\label{sec:openup}
In this section we compute the open GW invariants of $\mathcal{K}_{\mathbb{P}^1}\oplus\mathcal{O}_{\mathbb{P}^1}$. 
We give the space a $\C^*$ action with (Calabi-Yau) weights as in Figure \ref{fig:weights}.

\begin{figure}
	\centering
		\includegraphics[height=3.5cm]{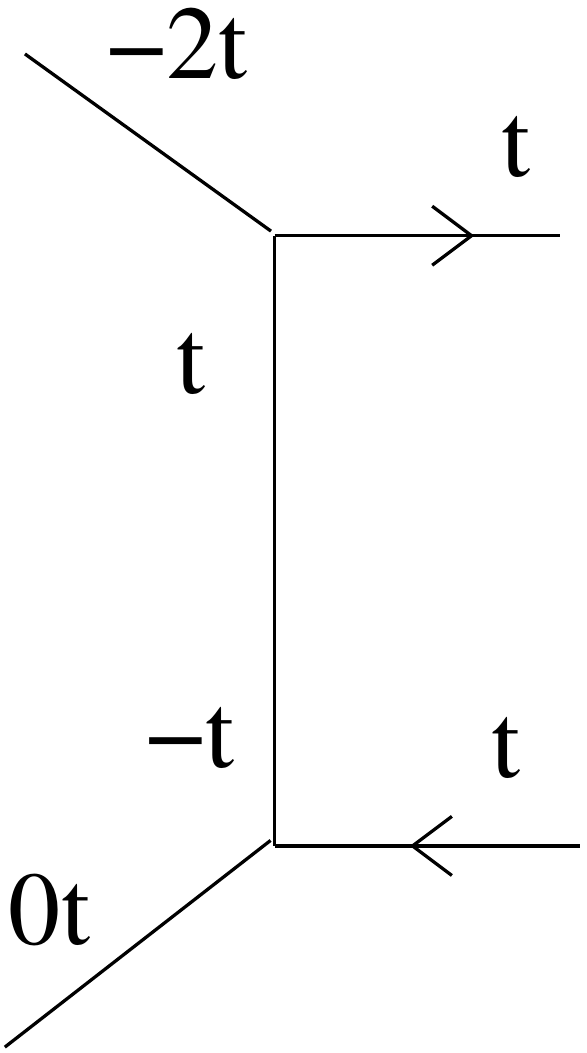}
	\caption{The web diagram for $\mathcal{K}_{\mathbb{P}^1}\oplus\mathcal{O}_{\mathbb{P}^1}$, and the specialized toric weights.}
	\label{fig:weights}
\end{figure}

In local coordinates at the top vertex, the action is defined by $\lambda\cdot (z,u,v)=(\lambda\cdot z, \lambda^{-2}\cdot u, \lambda\cdot v)$.  Similarly for the bottom vertex. The $\C^*$ fixed maps are quite easy to understand:

\begin{itemize}
\item The source curve consists of a genus $0$ (possibly nodal) closed curve along with attached disks.
\item The non-contracted irreducible components of the closed curve must be multiple covers of the torus invariant $\mathbb{P}^1$.
\item The disks must map to the fixed fibers of the trivial bundle with prescribed windings at the Lagrangians.
\end{itemize}

Analyzing the obstruction theory via the normalization sequence of the source curve, one sees that the $0$ weight at the bottom vertex limits the possible contributing maps in the following ways:

\begin{itemize}
\item Maps with positive dimensional components contracting to the bottom vertex do not contribute.
\item Maps with nodes mapping to the bottom vertex contribute only if the node connects a $d$-fold cover of the invariant $\mathbb{P}^1$ to a disk with winding $d$.
\end{itemize}

Fixed loci $F_{\Gamma}$ are indexed by localization graphs as in Figure \ref{localgraph}.  The combinatorial data is given by three multi-indices:
\begin{itemize}
\item $k_1,...,k_l$ the degrees of the multiple covers of the invariant $\mathbb{P}^1$ which do not attach to a disk at the bottom vertex.
\item $d_1,...,d_m$ the winding profile of the disks with origin mapping to the top vertex.
\item $d_{m+1},...,d_n$ the winding profile of the disks with origin mapping to the bottom vertex or equivalently if $n>1$ these are the degrees of the multiple covers of the invariant $\mathbb{P}^1$ which do attach to a disk at the bottom vertex.  
\item If $n=1$, we have the possibility of maps from a single disk mapping the origin to the bottom vertex, we label the locus of such maps $\Gamma'$.
\end{itemize} 

With the given multi-indices, the fixed locus $F_{\Gamma}$ is isomorphic to a finite quotient of $\overline{\mathcal{M}}_{0,n+l}$ where we interpret $\overline{\mathcal{M}}_{0,1}$ and $\overline{\mathcal{M}}_{0,2}$ as points.  Define the contribution from a fixed locus $\Gamma$ to be
\begin{equation}
OGW(\Gamma):=\int_{F_{\Gamma}}\frac{i^*[\overline{\mathcal{M}}]^{\text{vir}}}{e(N_{\text{vir}})}
\end{equation}
where $i^*[\overline{\mathcal{M}}]^{\text{vir}}$ is the restriction of the virtual fundamental class (proposed in \cite{kl:oinv}) to the fixed locus and $N_{\text{vir}}$ denotes the virtual normal bundle of $F_{\Gamma}$ in the moduli space of stable maps.

\begin{figure}
\includegraphics[height=3.5cm]{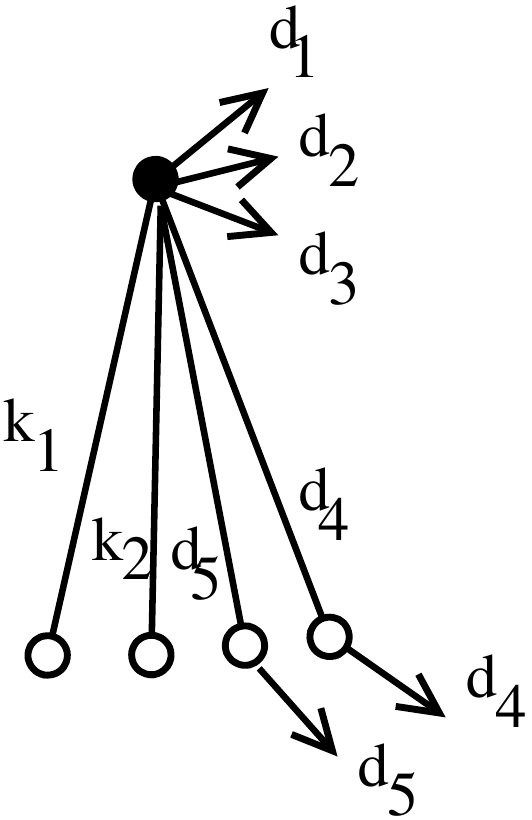}
\caption{The open localization graphs have bi-colored vertices to keep track of which vertex components contract to, and decorated arrows to represent disks mapping with given winding.}
\label{localgraph}
\end{figure}

In order to package the invariants in the Gromov-Witten potential, we assign the following formal variables:

\begin{itemize}
\item $q$ tracks the degree of the map on the base $\mathbb{P}^1$
\item $y_i^{(t)}$ tracks the number of disks with winding $i$ at the top vertex
\item $y_i^{(b)}$ tracks the number of disks with winding $i$ at the bottom vertex
\item $x$ tracks insertions of the nontrivial cohomology class (conveniently this class is a divisor).
\end{itemize}

The open potential is computed by adding the contributions of all fixed loci:
\begin{align}\label{OGW1}
OGW_{\mathcal{K}_{\mathbb{P}^1}\oplus\mathcal{O}_{\mathbb{P}^1}}(x,q,y_i^{(t)},y_i^{(b)})&=\sum_{\Gamma'}OGW(\Gamma')y_d^{(b)} \nonumber \\
&\hspace{-2cm}+\sum_{\Gamma\neq\Gamma'}OGW(\Gamma)(qe^x)^{k+d_{m+1}+...+d_n}y_{d_1}^{(t)}\cdot...\cdot y_{d_m}^{(t)}y_{d_{m+1}}^{(b)}\cdot...\cdot y_{d_n}^{(b)}
\end{align}

In (\ref{OGW1}), $\Gamma'$ denotes graphs consisting of a single white vertex and arrow labelled with winding $d$.
For non-degenerate graphs $\Gamma\not=\Gamma'$, we denote by $\overline{OGW}(\Gamma)$ the contribution to the potential from the fixed locus indexed by $\Gamma$, including invariants with any number of divisor insertions.
Following the obstruction theory for open invariants proposed in \cite{kl:oinv}, $OGW(\Gamma)$ are computed using the following ingredients: the euler class of the push-pull of the tangent bundle, the euler class of the normal bundle of $F_{\Gamma}$ in the moduli space of stable maps, and all relevant automorphisms of the map:
\begin{equation}
\frac{1}{|\text{glob. aut.}|}\int_{F_{\Gamma}}\frac{e(-R^\bullet\pi_*f^*T_{\mathcal{K}_{\mathbb{P}^1}\oplus\mathcal{O}_{\mathbb{P}^1}})\cdot(\text{inf. aut.})}{(\text{smoothing of nodes})}
\end{equation}

For convenience, we organize the computation on each locus $\Gamma$ into three parts:
\begin{itemize}
\item \textbf{Closed Curve:} This  consists of a closed curve contracting to the upper vertex as well as multiple covers of the torus fixed $\mathbb{P}^1$.  We choose not to include the $d$-covers of the fixed line which are attached to a disk mapping with winding $d$ to the bottom vertex.  The contracted component contributes $(-2t^3)^{-1}$ from the push-pull of the tangent bundle and each $k$-cover contributes \begin{equation}\frac{-t}{k^2}\frac{eH^1(\mathcal{O}(-2k))}{eH^0(\mathcal{O})eH^0(\mathcal{O}(2k))}=\frac{(-1)^{k}}{tk^2}{2k-1 \choose k}.\end{equation} Here we have included both the global automorphism of the $k:1$ cover and the infinitesimal automorphism at the point ramified over the bottom vertex.
\item \textbf{Disks:}  A disk can either be mapped to the top or the bottom vertex.  Following Katz and Liu \cite{kl:oinv}, the contribution of a disk mapping to the top vertex with winding $d$ is given by
\begin{equation}
\frac{1}{d}\frac{eH^1(N(d))}{eH^0(L(2d))}=\frac{(-1)^{d+1}}{td}{2d-1 \choose d}
\end{equation}
where $L(2d)$ and $N(d)$ are defined in Examples 3.4.3. and 3.4.4 of \cite{kl:oinv}.  We have divided the contributions in a way that the contribution of a disk mapping to the bottom vertex also includes the contribution of the multiple cover attaching it to the contracted component.  The reason for this is that the combined contribution becomes
\begin{align}
&\frac{1}{d^2}\frac{eH^1(\mathcal{O}(-2d))}{eH^0(\mathcal{O})eH^0(\mathcal{O}(2d))}\frac{eH^1(N(d))}{eH^0(L(2d))}\frac{eH^0(N_{/X})}{\frac{t}{d}-\frac{t}{d}} \nonumber \\
&=\frac{1}{d^2}\frac{(-1)^{d+1}}{t^2}{2d-1\choose d}\frac{1}{t}\frac{-0t^3}{\frac{t}{d}-\frac{t}{d}} \nonumber \\
&=\frac{(-1)^{d+1}}{td}{2d-1\choose d}
\end{align}
which is the same as the contribution of the disk at the top vertex.  

\begin{remark}
In order to interpret the expression $\frac{-0}{1-1}$ in the above equations, recall that it arises as $\frac{s_1s_2s_3}{s_1+s_2}$ where the $s_i$ sum to $0$.  As $s_3\rightarrow 0$, the quotient tends to $-s_1s_2$.
\end{remark}

\item \textbf{Nodes:} Since we have already accounted for the nodes at the bottom vertex (those attaching winding $d$ disks to $d:1$ covers), this piece  only contains the contribution from nodes at the top vertex.  For each such node connecting either a disk of winding $d$ or a curve of degree $d$ to the contracted component we  get a contribution of $-2t^3$ from the push-pull of the tangent sheaf and a contribution of $\frac{1}{(\frac{t}{d}-\psi_i)}$ from node smoothing.
\end{itemize}

Putting the pieces together:

\begin{align}
\label{Step2}
OGW(\Gamma)=&\frac{1}{|\text{Aut}(\Gamma)|}\prod_{i=1}^l\frac{(-1)^{k_i}}{tk_i^2}{2k_i-1 \choose k_i}\prod_{i=1}^n\frac{(-1)^{d_i+1}}{td_i}{2d_i-1 \choose d_i} \nonumber\\
&\hspace{1cm}\cdot(-2t^3)^{l+n-1}\int_{\overline{M}_{0,n+l}}\frac{1}{\prod(\frac{t}{k_i}-\psi_i)\prod(\frac{t}{d_i}-\psi_{i+l})}.
\end{align}
where Aut$(\Gamma)$ is the product of the automorphisms of the ordered tuples $(k_1,...,k_l)$, $(d_1,...,d_m)$, and $(d_{m+1},...,d_n)$. 

Applying the string equation to the integral and simplifying, (\ref{Step2}) becomes

\vspace{0.3cm}
$OGW(\Gamma) = $
\begin{equation}\label{OGW2}
 \frac{-2^{l+n-1}}{|\text{Aut}(\Gamma)|}\Bigg[\prod_{i=1}^l\frac{(-1)^{k_i+1}}{k_i}{2k_i-1 \choose k_i}\Bigg]\Bigg[\prod_{i=1}^n(-1)^{d_i}{2d_i-1 \choose d_i}\Bigg](d+k)^{l+n-3}
\end{equation}
where $d=\sum d_i$ and $k=\sum k_i$.  

Recall now that the contribution of a disk is the same regardless of whether it maps to the top or bottom Lagrangian.  Therefore, if we let $\Gamma(\bar{d};\bar{k})$ denote all $\Gamma\neq\Gamma'$ with winding profile $\bar{d}=(d_1,...,d_n)$ and fixed $\bar{k}=(k_1,...,k_l)$, we can attach the formal variables and compute:

\begin{align}
\label{undo}
&\sum_{\Gamma\in\Gamma(\bar{d};\bar{k})}\overline{OGW}(\Gamma)=\frac{-2^{l+n-1}}{|\text{Aut}(\overline{d})|}\prod_{i=1}^n\left(y_{d_i}^{(t)}+y_{d_i}^{(b)}(qe^x)^{d_i}\right)\prod_{i=1}^n(-1)^{d_i}{2d_i-1 \choose d_i}\nonumber \\
&\hspace{1cm}\cdot\frac{1}{|\text{Aut}(\overline{k})|}\prod_{i=1}^l\frac{(-1)^{k_i+1} (qe^x)^{k_i}}{k_i}{2k_i-1 \choose k_i}(d+k)^{l+n-3}
\end{align}

We now sum over all $\bar k$ with $\sum k_i=k$.  In order to do this, set
\begin{align}
F(X,Y)  := \exp\left(\sum_{\kappa\geq 1}\frac{(-1)^{\kappa+1}}{\kappa}{2\kappa-1 \choose \kappa}X^\kappa Y\right)\nonumber\\
   =  \sum_{l,k}\sum_{\overline{k}}\frac{1}{|\text{Aut}(\overline{k})|}\left[\prod_{i=1}^l \frac{(-1)^{k_i+1}}{k_i}{2k_i-1 \choose k_i}\right]X^kY^l
\end{align}
where the second sum is over all $l$-tuples $\overline{k}=(k_1,...,k_l)$ with $\sum k_i=k$.  The sum of all contributions with fixed winding profile $(d_1,...,d_n)$ and with $(k_1,...,k_l)$ satisfying $\sum k_i=k$ is obtained by specializing $Y=2(d+k)$ and multiplying the coefficient of $X^k$ by an appropriate factor:  
\begin{align}\label{OGW4}
\sum_{|\bar{k}|=k}\sum_{\Gamma\in\Gamma(\bar{d};\bar{k})}\overline{OGW}(\Gamma)=\frac{-2^{n-1}}{|\text{Aut}(\bar{d})|}\prod_{i=1}^n\left(y_{d_i}^{(t)}+y_{d_i}^{(b)}(qe^x)^{d_i}\right)\prod_{i=1}^n(-1)^{d_i}{2d_i-1 \choose d_i}\nonumber \\
 \hspace{2cm} \cdot (qe^x)^k(d+k)^{n-3} [F(X,2(d+k))]_{X^k}.
\end{align}


To handle (\ref{OGW4}), we find a closed form  expression for $F$. Start with the known generating function
\begin{equation}
\sum_{k\geq 1}{2k-1 \choose k}(-1)^kX^k=\frac{1}{2}\cdot\frac{1-\sqrt{1+4X}}{\sqrt{1+4X}}
\end{equation}
If we divide by $-X$ and formally integrate term by term (imposing that the constant term is $0$), we get
\begin{equation}
\sum_{k\geq 1}\frac{(-1)^{k+1}}{k}{2k-1 \choose k}X^k=\ln\left(\frac{1}{2}(1+\sqrt{1+4X})\right)
\end{equation}
Finally, we can write 
\begin{align}
F&= \exp\left(Y\ln\left(\frac{1}{2}(1+\sqrt{1+4X})\right)\right) \nonumber\\
&=\left[\frac{1}{2}(1+\sqrt{1+4X})\right]^Y
\end{align}

There are a few interesting comments to make at this point:
\begin{itemize}
\item Setting $G:=\frac{1}{2}(1+\sqrt{1+4X})$, we see that $G=1+X\cdot C(X)$ where $C(X)$ is the generating function for the Catalan numbers.
\item $G$ satisfies the recursive relation $G^n=G^{n-1}+XG^{n-2}$.
\item It is easy to see that the recursion and the relation between $G$ and the Catalan numbers are equivalent to the array of coefficients of $G^i$ taking on a slight variation of two classical combinatorial objects, as illustrated in Figure \ref{fig:catalan}. Here ``slight variation'' is probably best described by looking at the first few terms in Table \ref{tab:catalan}.

\begin{figure}[b]
	\centering
		\includegraphics[width=0.35\textwidth]{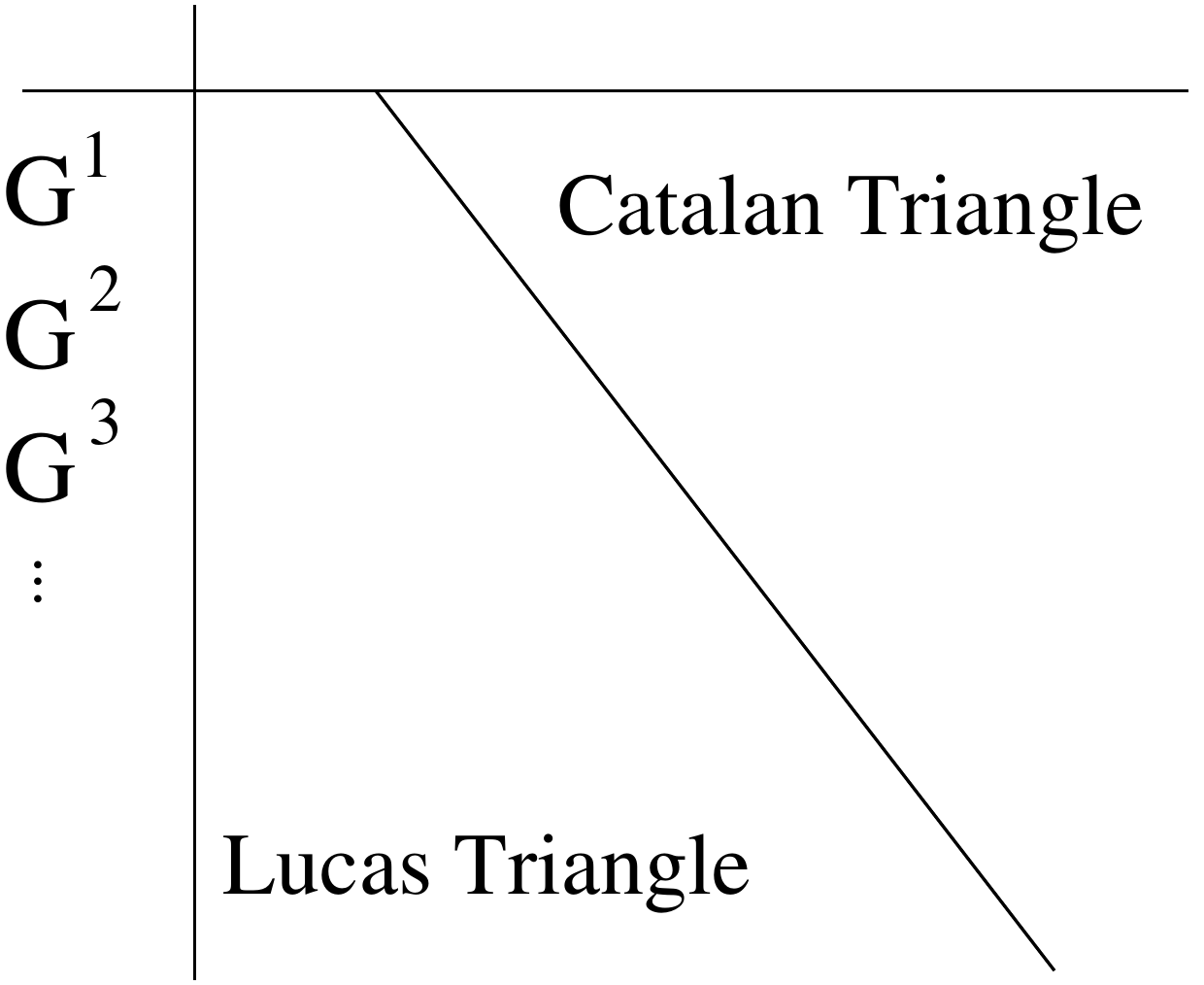}
	\caption{The coefficients of $G^n$ as classical combinatorial numbers.}
	\label{fig:catalan}
\end{figure}


\begin{table}
$$
\begin{array} {ccccccccccc}
& 1 & x & x^2 & x^3 & x^4 & x^5 & x^6 & x^7 & x^8 & x^9 \\
G & 1 & 1 & -1 & 2 & -5 & 14 & -42 & 132 & -429 & 1430\\
G^2 & 1&2&-1&2&-5&14&-42&132&-429&1430\\
G^3&1&3&0&1&-3&9&-28&90&-297&1001\\
G^4&1&4&2&0&-1&4&-14&48&-165&572\\
G^5&1&5&5&0&0&1&-5&20&-75&275\\
G^6&1&6&9&2&0&0&-1&6&-27&110\\
G^7&1&7&14&7&0&0&0&1&-7&35\\
G^8&1&8&20&16&2&0&0&0&-1&8\\
G^9&1&9&27&30&9&0&0&0&0&1\\
G^{10}&1&10&35&50&25&2&0&0&0&0\\
\end{array}
$$
\caption{The first coefficients of the series of $G^n$.}
\label{tab:catalan}
\end{table}
\end{itemize}

Using the recursion and induction, one easily proves the following lemma.

\begin{lemma}
If $d>0$, the $X^k$ coefficient of $G^{2(d+k)}$ is \begin{equation}{k+(2d-1) \choose 2d-1}\frac{d+k}{d}.\end{equation} The $X^k$ coefficient of $G^{2k}$ is $2$.
\end{lemma}

These are precisely the coefficients we need.  Therefore, we can conclude:
\begin{itemize}
\item From equation (\ref{OGW4}), if $(d_1,...,d_n)\neq\emptyset$, then 
\begin{align}
\sum_{|\bar{k}|=k}&\sum_{\Gamma\in\Gamma(\bar{d};\bar{k})}\overline{OGW}(\Gamma)=\frac{-2^{n-1}}{d\cdot|\text{Aut}(\overline{d})|}\prod_{i=1}^n\left(y_{d_i}^{(t)}+y_{d_i}^{(b)}(qe^x)^{d_i}\right)\nonumber \\
& \cdot\prod_{i=1}^n(-1)^{d_i}{2d_i-1 \choose d_i}\sum_{k\geq 0} (d+k)^{n-2}  {k+(2d-1) \choose 2d-1} (qe^x)^k. 
\end{align}
\item Also from equations  (\ref{OGW4}), if $(d_1,...,d_n)=\emptyset$ and $(k_1,...,k_l)\neq\emptyset$, then
\begin{equation}
\sum_{|\bar{k}|=k} \sum_{\Gamma\in\Gamma(\emptyset;\bar{k})}\overline{OGW}(\Gamma)=\frac{-1}{k^3}(qe^x)^k.
\end{equation}
Here we have recovered the Aspinwall-Morrison formula for $\mathcal{K}_{\mathbb{P}^1}\oplus\mathcal{O}_{\mathbb{P}^1}$.
\end{itemize}
Finally recall that:
\begin{itemize}
\item If both $\bar{d}=\emptyset$ and $\bar{k}=\emptyset$, then the locus consists of the degree $0$ maps with only divisor insertions which can be computed via localization to be
\begin{equation}
\frac{-x^3}{12}.
\end{equation}
\item The contribution from a locus $\Gamma'$ consisting of a single disk mapping to the bottom vertex with winding $d$ is given by
\begin{equation}
\frac{1}{d^2}y_d^{(b)}.
\end{equation}
\end{itemize}

Adding all contributions we conclude that
\begin{align}
\label{almostthere}
&OGW_{\mathcal{K}_{\mathbb{P}^1}\oplus\mathcal{O}_{\mathbb{P}^1}}(x,q,y_i^{(t)},y_i^{(b)})=\frac{-1}{2}\frac{x^3}{3!}+\sum_{k\geq 1}\frac{-1}{k^3}(qe^x)^k + \sum_{d\geq 1} \frac{1}{d^2}y_d^{(b)}\nonumber  \\
&+\sum_{(d_1,...,d_n)\neq\emptyset} \Bigg[ \frac{-2^{n-1}}{d\cdot|\text{Aut}(\overline{d})|}\prod_{i=1}^n\left(y_{d_i}^{(t)}+y_{d_i}^{(b)}(qe^x)^{d_i}\right)\prod_{i=1}^n(-1)^{d_i}{2d_i-1 \choose d_i}\nonumber \\ 
& \hspace{1cm}\cdot \sum_{k\geq 0} (d+k)^{n-2}  {k+(2d-1) \choose 2d-1} (qe^x)^k \Bigg]
\end{align}

In a neighborhood of $x=-\infty$ we have:
\begin{equation}
\label{resum}
\sum_{k\geq 0} {k+(2d-1) \choose 2d-1} (qe^x)^{d+k}=\frac{(qe^x)^d}{(1-qe^x)^{2d}}.
\end{equation}
Using (\ref{resum}) we can express (\ref{almostthere}):
\begin{equation}
\sum_{k\geq 0} (d+k)^{n-2} {k+(2d-1) \choose 2d-1} (qe^x)^k=\frac{1}{(qe^x)^d}\frac{d^{n-2}}{dx^{n-2}}\left(\frac{(qe^x)^d}{(1-qe^x)^{2d}}\right)
\end{equation}
where differentiation/integration is computed formally termwise.  When $n\geq 2$, there is no ambiguity as $\frac{d^{n-2}}{dx^{n-2}}$ is a derivative.  When $n=1$, we must practice a little bit of caution as the integral is only defined up to translation.  
Notice that
\begin{equation}
\lim_{x\rightarrow -\infty}\sum_{k\geq 0}\frac{1}{k+d}{k+(2d-1) \choose 2d-1}(qe^{x})^{k+d}=0,
\end{equation}

hence by 
$$
\int\frac{(qe^{x})^d}{(1-qe^{x})^{2d}}dx
$$
we denote the antiderivative having limit $0$  as $x$ approaches $-\infty$.


We conclude this section by putting the open potential in its simplest form:

\begin{proposition}\label{OGWres}
The open Gromov-Witten potential (sans fundamental class insertions) for $\mathcal{K}_{\mathbb{P}^1}\oplus\mathcal{O}_{\mathbb{P}^1}$ is
\begin{align*}
&OGW_{\mathcal{K}_{\mathbb{P}^1}\oplus\mathcal{O}_{\mathbb{P}^1}}(x,q,y_i^{(t)},y_i^{(b)})\hspace{.3cm}=\hspace{.3cm}\frac{-1}{12}x^3+\sum_{k\geq 1}\frac{-1}{k^3}(qe^x)^k\\
&+\hspace{.5cm}\sum_{d\geq 1} \Bigg[\frac{1}{d^2}y_d^{(b)} +\frac{(-1)^{d+1}}{d} \left(y_{d}^{(t)}+y_{d}^{(b)}(qe^x)^{d}\right){2d-1 \choose d}\\
&\hspace{3cm}\cdot\frac{1}{(qe^x)^d}\int \frac{(qe^{x})^d}{(1-qe^{x})^{2d}}dx \Bigg] \\
&+\sum_{d_1,...,d_n (n\geq 2)} \Bigg[ \frac{-2^{n-1}}{d\cdot|\text{Aut}(\overline{d})|}\left[\prod_{i=1}^n(-1)^{d_i}\left(y_{d_i}^{(t)}+y_{d_i}^{(b)}(qe^x)^{d_i}\right){2d_i-1 \choose d_i}\right]\\ 
& \hspace{3cm}\cdot \frac{1}{(qe^x)^d}\frac{d^{n-2}}{dx^{n-2}}\left(\frac{(qe^x)^d}{(1-qe^x)^{2d}}\right) \Bigg].
\end{align*}
\end{proposition}

The first line is the closed contribution, the next two lines are the contribution from curves with one boundary component, and the final two lines are the contribution from curves with more than one boundary component.


\section{Open Orbifold Gromov-Witten Invariants of $[\C^3/\Z_2]$}\label{sec:opendown}
\label{sec:oogwi}

In this section we compute the open orbifold GW invariants of $[\C^3/\Z_2]$ following \cite{bc:ooinv}.  We define a $\C^*$ action on the orbifold with weights described in Figure \ref{fig:orbvert}:


\begin{figure}
	\centering
		\includegraphics[height=1.5cm]{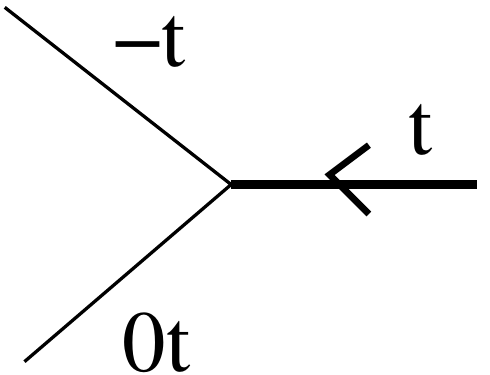}
	\caption{Toric diagram for $[\C^3/\Z_2]$ and $\C^\ast$ weights.}
	\label{fig:orbvert}
\end{figure}

We characterize the $\C^*$ fixed maps:
\begin{itemize}
\item The source curve consists of a genus $0$  closed curve along with attached disks.  The closed component can carry (possibly twisted) marks whereas a disk can only carry a mark at the origin (if it is not attached to a closed component).  The attaching points of the nodes must carry inverse twisting.
\item The closed curve must contract to the vertex.
\item The disks must map to the twisted $\C$ with prescribed windings at the Lagrangian.
\end{itemize}

Since we are working with a $\Z_2$ quotient, we  simply refer to points as twisted or untwisted as there is no ambiguity.  A careful analysis of the obstruction theory via the normalization sequence of the source curve shows that the $0$ weight conveniently kills all contributions where a disk attaches to a contracted component at an untwisted node.  By dimensional reasons, all other marks must be twisted.

Combinatorially, the fixed loci $\Lambda$ are indexed by
\begin{itemize}
\item $m$ the number of insertions of the twisted sector and
\item $d_1,...,d_n$ the winding profile of the disks.
\end{itemize}

\begin{remark}\label{even}
Since all nodes and marked points are twisted, the maps restricted to the contracted component (maps into $\mathcal{B}\Z_2$) classify double covers of the contracted component with simple ramification over $m+n$ points.  Since such a cover only exists if $m+n$ is even, the loci are non-empty only when $m+n$ is even.
\end{remark}

If we let $z$ and $w_d$ be formal variables tracking the twisted sector insertions and the winding $d$ disks, then the open orbifold potential can be computed as

\begin{equation}
OGW_{[\C^3/\Z_2]}(z,w_i)=\sum_{\Lambda}OGW(\Lambda)\frac{z^m}{m!}w_{d_1}\cdot...\cdot w_{d_n}
\end{equation}

We now group the computation of $OGW(\Lambda)$ into three components:
\begin{itemize}
\item \textbf{Closed Curve:} The closed curve contracted to the vertex essentially carries the information of a map into $\mathcal{B}\Z_2$ along with the weights of the $\C^*$ action on the three normal directions.  This classifies a double cover of the source curve.  Analogous to \cite[section 2.1]{cc:gl}, the contribution from the closed component is the equivariant euler class of two copies of the dual of the Hodge bundle on the cover twisted by the weights of the action on the untwisted fixed fibers:
\begin{equation}
e(\mathbb{E}_{-1}^\vee(-1)\oplus\mathbb{E}_{-1}^\vee(0))
\end{equation}
\item \textbf{Disks:} The disk contribution is laid out in \cite[section 2.2.3]{bc:ooinv}.  This contribution is a combinatorial function depending on the winding at the Lagrangian and the twisting at the origin of the disk.  The localization step simplifies the disk contribution to two cases, either the origin of the disk is marked and twisted (possibly a node) or the origin is unmarked.  For the particular case at hand, a disk with winding $d$ and with twisting at the origin contributes
\begin{equation}
\frac{1}{2d}\frac{(2d-1)!!}{(2d)!!}
\end{equation}
whereas a disk with no mark and no twisting at the origin contributes
\begin{equation}
\frac{1}{2d^2}.
\end{equation}

\item \textbf{Nodes:} We consider the nodes attaching a winding $d$ disk to the closed component.  Each one  gets a $t$ from the weight of the action on the twisted sector.  Smoothing the node contributes $\frac{1}{\frac{t}{2d}-\frac{\psi_i}{2}}$.  
\end{itemize}

Putting together the three parts described above, we find that $OGW(\Lambda)$ is:
\begin{align*}
&\frac{1}{|\text{Aut}(\overline{d})|}\Bigg[\prod_{i=1}^n\frac{1}{2d}\frac{(2d_i-1)!!}{(2d_i)!!}\Bigg]\int(2t)^n\frac{e^{eq}(\mathbb{E}_{-1}^\vee(-1)\oplus\mathbb{E}_{-1}^\vee(0))}{\prod_{i=1}^n(\frac{t}{d_i}-\psi_i)}\\
&=\frac{1}{|\text{Aut}(\overline{d})|}\Bigg[\prod_{i=1}^n\frac{(2d_i-1)!!}{(2d_i)!!}\Bigg]\sum_{i=1}^{g-1}\sum_{|\overline{j}|=i-1}\int\lambda_g\lambda_{g-i}(\overline d\overline\psi)^{\overline{j}}
\end{align*}
where the integral is taken over $\overline{\mathcal{M}}_{0;m+n,0}(\mathcal{B}\Z_2)$, $g=\frac{m+n-2}{2}$ (the genus of the cover of the closed curve) and $(\overline d\overline\psi)^{\overline{j}}$ and $|\overline j|$ are defined in section \ref{sec:hodge}.

Summing over all $m$ (equivalently $g$) and specializing $q_i=d_i$ in Thereom \ref{gen}, we see that the contribution to the open potential from all maps with a fixed winding profile $d_1,...,d_n$ is given by
\begin{equation}
\frac{1}{|\text{Aut}(\overline{d})|}\Bigg[\prod_{i=1}^n\frac{(2d_i-1)!!}{(2d_i)!!}\Bigg]\frac{d^{n-2}}{dz^{n-2}}\frac{\sec^{2d}(z/2)}{2d}
\end{equation}
There is no ambiguity for $n\geq 2$, but we must again be careful when $n<2$.  

When $n=1$ the above formula still holds, but since integrals are only defined up to translation, we must make sure and get the correct constant term.  The constant term corresponds to the contribution from maps with one boundary component and no marked points.  The only type of map in the fixed locus that satisfies this criteria is a disk with no marked points mapping with winding $d$.  We've seen that the contribution from such a map is $\frac{1}{2d^2}$. 

When $n=0$, we must compute the closed contribution.  The maps must have at least 3 marked points to be stable, but any map into $\mathcal{B}\Z_2$ must have an even number of twisted points (see Remark \eqref{even}).  Since there are no disk or node smoothing factors, the contribution is 
\begin{equation}
H(z)=\sum_{g\geq 1}\int_{\overline{\mathcal{M}}_{0;2g+2,0}(\mathcal{B}\Z_2)}\lambda_g\lambda_{g-1}\frac{z^{2g+2}}{(2g+2)!}
\end{equation}
and the now classical $\lambda_g\lambda_{g-1}$ result of Faber and Pandharipande \cite{fp:lsahiittr} implies that $\frac{d^2}{dz^2}H(z)=\log(\sec(z/2)).$

Pulling together everything from the above discussion, we prove the following result:

\begin{proposition}
\label{prop:orboppot}
The open orbifold Gromov-Witten potential (sans fundamental class insertions) of $[\C^3/\Z_2]$ is
\begin{align}
&OGW_{[\C^3/\Z_2]}(z,w_i)=\hspace{.5cm}H(z)\nonumber\\
& + \hspace{.5cm} \sum_{d\geq 1}\left(\frac{1}{2d^2}+\frac{(2d-1)!!}{(2d)!!}\int\frac{\sec^{2d}(z/2)}{2d}dz\right)w_d \nonumber \\
& + \hspace{-.2cm} \sum_{d_1,...,d_n(n\geq 2)} \frac{1}{|\text{Aut}(\overline{d})|}\Bigg(\prod_{i=1}^n\frac{(2d_i-1)!!}{(2d_i)!!}\Bigg)\left(\frac{d^{n-2}}{dz^{n-2}}\frac{\sec^{2d}(z/2)}{2d}\right)w_{d_1}\cdot...\cdot w_{d_n},
\end{align}
where the antiderivative is chosen to vanish at $z=0$.
\end{proposition}

\section{An Example of the Open Crepant Resolution Conjecture}\label{sec:openres}
\label{sec:five}

Now that we have computed the open potentials for $[\C^3/\Z_2]$ and its crepant resolution $\mathcal{K}_{\mathbb{P}^1}\oplus\mathcal{O}_{\mathbb{P}^1}$, we show that there is a change of variables which equates the stable terms of the two potentials.  We start with the contribution from a given winding profile on the orbifold, we consider all contributions on the resolution with that same winding profile, and we show that the change of variables equates these contributions.  More specifically, we show the following.

\begin{theorem}
\label{thm:ocrc}
Under the change of variables
\begin{align}
q&\rightarrow -1 \nonumber\\
x&\rightarrow iz\nonumber\\
y_d^{(b)}&\rightarrow \frac{i}{2}w_d\nonumber\\
y_d^{(t)}&\rightarrow \frac{i}{2}w_d(-e^{iz})^d
\end{align}
the open GW potential of $\mathcal{K}_{\mathbb{P}^1}\oplus\mathcal{O}_{\mathbb{P}^1}$ analytically continues to the open GW potential of $[\C^3/\Z_2]$ up to unstable terms.

\end{theorem}

\begin{proof}On the closed portion of the potential, we essentially (up to a harmless weight factor) have the result of \cite[Section 3.2]{bg:crc}.


Now consider the winding $d$ disk contribution on the resolution:  
\begin{align}
&\frac{1}{d^2}y_d^{(b)} +\frac{(-1)^{d+1}}{d} \left(y_d^{(t)}+y_d^{(b)}(qe^x)^d)\right){2d-1 \choose d} \nonumber\\
&\hspace{3cm}\cdot\frac{1}{(qe^x)^d}\Bigg[\int \frac{(qe^{x})^d}{(1-qe^{x})^{2d}}dx \Bigg].
\end{align}
Making the change of variables, it becomes
\begin{align*}
&\frac{i}{2d^2}w_d +\frac{(-1)^{d+1}}{d} \left(\frac{i}{2}w_d(-e^{iz})^d+\frac{i}{2}w_d(-e^{iz})^d)\right){2d-1 \choose d}\\
&\hspace{2cm}\cdot\frac{1}{(-e^{iz})^d}\Bigg[i\int \frac{(-e^{iz})^d}{(1+e^{iz})^{2d}}dz \Bigg]\\
&=\frac{i}{2d^2}w_d +\frac{i(-1)^{d+1}}{d} w_d{2d-1 \choose d}
\Bigg[i\int \frac{(-e^{iz})^d}{(1+e^{iz})^{2d}}dz \Bigg]\\
&=\frac{i}{2d^2}w_d +\frac{-i}{d} w_d{2d-1 \choose d}
\cdot\Bigg[i\int \frac{\sec^{2d}(z/2)}{(2^{2d})}dz \Bigg]\\
\end{align*}


Here we do not pay attention to the constant terms in the anti-derivatives since they correspond to unstable terms about which we make no claims. Hence we obtain: 

\begin{equation}
\frac{1}{d}{2d-1 \choose d}w_d\int \frac{\sec^{2d}(z/2)}{(2^{2d})}dz=\frac{(2d-1)!!}{(2d)!!}w_d\int\frac{\sec^{2d}(z/2)}{2d}dz, 
\end{equation}
the disk potential computed on the orbifold.

Finally, consider a general term in the open potential of the resolution with winding profile $d_1,...,d_n$:
\begin{equation}
 \frac{-2^{n-1}}{d\cdot|\text{Aut}(\overline{d})|}\left[\prod_{i=1}^n(-1)^{d_i}\left(y_{d_i}^{(t)}+y_{d_i}^{(b)}(qe^x)^{d_i}\right){2d_i-1 \choose d_i}\right]
 \frac{1}{(qe^x)^d}\frac{d^{n-2}}{dx^{n-2}}\left(\frac{(qe^x)^d}{(1-qe^x)^{2d}}\right)
\end{equation}
Making the change of variables, this becomes
\begin{align*}
&\frac{-2^{n-1}}{d\cdot |\text{Aut}(\overline{d})|}{(i)}^n\prod_{i=1}^nw_{d_i}{2d_i-1 \choose d_i}\frac{1}{i^{n-2}}\frac{d^{n-2}}{dz^{n-2}}\frac{1}{2^{2d}}\sec^{2d}\left(\frac{z}{2}\right)\\
&=\frac{1}{2d\cdot |\text{Aut}(\overline{d})|}\prod_{i=1}^n\frac{w_{d_i}}{2^{2d_i-1}}{2d_i-1 \choose d_i}\left(\frac{d^{n-2}} {dz^{n-2}}\sec^{2d}\left(\frac{z}{2}\right)\right)\\
&=\frac{1}{|\text{Aut}(\overline{d})|}\prod_{i=1}^n\frac{(2d-1)!!}{(2d)!!}\left(\frac{d^{n-2}} {dz^{n-2}}\frac{\sec^{2d}\left(\frac{z}{2}\right)}{2d}\right)w_{d_i}\cdot ... \cdot w_{d_n}
\end{align*}

and this final expression coincides with the contribution on the orbifold.
\end{proof}

\section{Gluing Open Invariants}\label{sec:glue}

In this section we develop rules for gluing open GW invariants to obtain closed GW invariants.  For non-orbifold invariants, we develop a general rule for gluing invariants from trivalent vertices with any compatible torus actions.  For orbifold invariants, we specialize to the case of the $\Z_2$ quotient with the specific torus action introduced in the previous sections.

\subsection{Non-orbifold Gluing}
\label{sec:gl}

In the spirit of the topological vertex [AKMV], we show in this section that the open invariants defined by Katz and Liu can be glued to obtain closed invariants of a smooth toric Calabi-Yau threefold.  Any smooth toric Calabi-Yau threefold can be equipped with a $\C^*$ action so that the three weights at any vertex of the web diagram sum to zero (Calabi-Yau weights).  The torus action can be lifted to the moduli space of stable maps and the fixed loci consist of maps which contract components to the vertices and map rational components to the compact edges via multiple covers fully ramified over the vertices.  The Gromov-Witten potential is then computed as a sum over contributions coming from these fixed loci.
  
Placing a Lagrangian along each compact edge of the web diagram, we can ``cut'' each fixed locus into a locus of open maps at each vertex.  In this section, we show that the contribution of the fixed locus to the usual Gromov-Witten potential can be obtained essentially by multiplying the corresponding open Gromov-Witten invariants.  The standard procedure for localization computations of Gromov-Witten invariants shows that the only thing we need to check is that the contribution from a multiple cover of a compact edge can be recovered from the disk contributions on each half-edge.  Specifically, we show that the degree $d$ multiple cover contribution of $\mathcal{O}_{\mathbb{P}^1}(-k)\oplus\mathcal{O}_{\mathbb{P}^1}(k-2)$ can be obtained essentially by multiplying winding $d$ disk contributions on each of the vertices.
  
\begin{proposition}
\label{prop:regglue}
Closed GW invariants of a smooth toric Calabi-Yau threefold are obtained by computing the open invariants at each vertex and then contract winding $d$ contributions along the edges with a factor of \\
$\begin{cases}
(-1)^{dk+1}d &\text{if the half-edges have the same orientation}\\
(-1)^{dk+d}d  &\text{if the half-edges have opposite orientation.}
\end{cases}$
\end{proposition}  


\begin{figure}
	\centering
		\includegraphics[height=2.3cm]{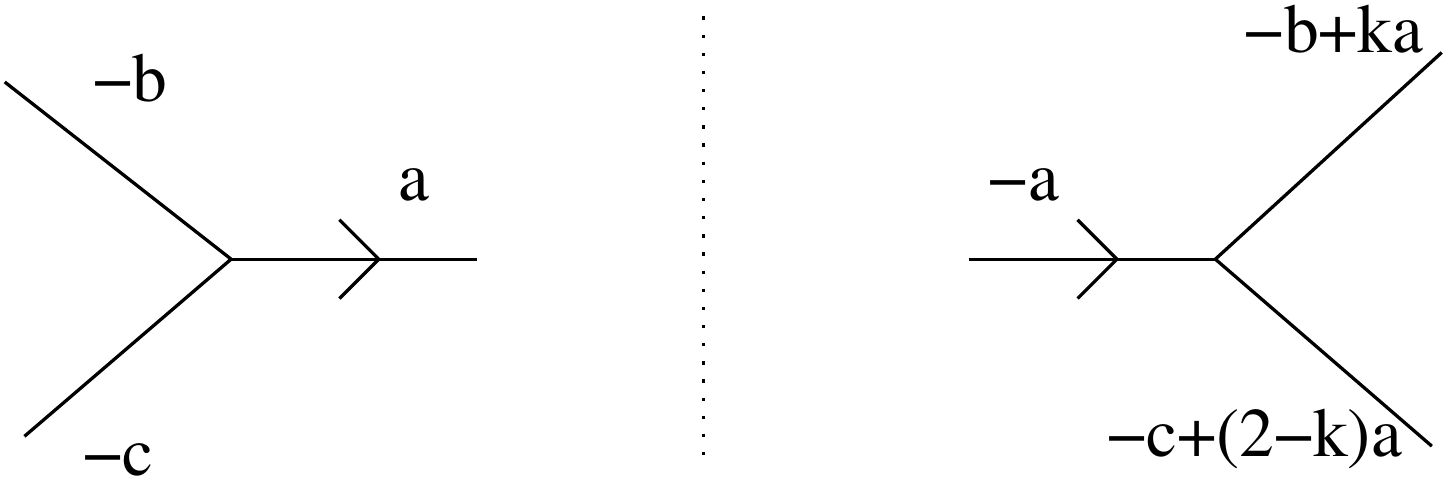}
	\caption{A general edge in the web diagram of a toric Calabi-Yau has normal bundle $\mathcal{O}_{\mathbb{P}^1}(-k)\oplus\mathcal{O}_{\mathbb{P}^1}(k-2)$.}
	\label{fig:disks}
\end{figure}

\begin{proof}
Figure \ref{fig:disks} gives arbitrary Calabi-Yau weights for a neighborhood of a general fixed line in a toric Calabi-Yau.
Assume $k$ and $a$ are positive.  One computes the winding $d$ disk invariant on the left vertex to be:
\begin{equation}
\frac{(-1)^{d+1}}{a^{d-1}d(d!)}\prod_{i=1}^{d-1}(bd-ai)
\end{equation}
and the winding $d$ disk invariant on the right vertex is:
\begin{equation}
\frac{1}{a^{d-1}d(d!)}\prod_{i=1}^{d-1}(bd-akd+ai).
\end{equation}

If we now multiply the two disk invariants, we get:
\begin{equation}
\label{dg}
\frac{(-1)^{d+1}}{a^{2d-2}d^2(d!)^2}\prod_{i=1}^{d-1}(bd-ai)(bd-akd+ai)
\end{equation}
We now compare \eqref{dg} with the contribution to the closed GW potential of $\mathcal{O}_{\mathbb{P}^1}(-k)\oplus\mathcal{O}_{\mathbb{P}^1}(k-2)$ given by a degree $d$ multiple cover. In the case $k=1$ we get immediately the same expression (up to the appropriate sign factor). For $k\geq 2$, the contribution is: 

\begin{equation}
\label{tc}
\frac{(-1)^{d+1}}{a^{2d-2}d(d!)^2}\frac{\prod_{i=1}^{dk-1}(-bd+ai)}{\prod_{i=0}^{d(k-2)}(d(b-a)-ia)}
\end{equation}

We can use the fact that 
\begin{equation}
bd-ai=d(b-a)-ja\Longleftrightarrow i=d+j.
\end{equation}
to write (\ref{tc}) as:

\begin{align*}
&\frac{(-1)^{d+1}(-1)^{dk-1}}{a^{2d-2}d(d!)^2}\prod_{i=1}^{d-1}(bd-ai)\frac{\prod_{i=d}^{dk-d}(bd-ai)}{\prod_{j=0}^{d(k-2)}(d(b-a)-ja)}\prod_{i=dk-d+1}^{dk-1}(bd-ai)\\
&=\frac{(-1)^{d+1}(-1)^{dk-1}}{a^{2d-2}d(d!)^2}\prod_{i=1}^{d-1}(bd-ai)(bd-a(dk-d+i))
\end{align*}

Reversing the index on the second term in the product:
\begin{equation}
\label{gw}
\frac{(-1)^{d(k+1)}}{a^{2d-2}d(d!)^2}\prod_{i=1}^{d-1}(bd-ai)(bd-akd+ai).
\end{equation}

Comparing (\ref{dg}) with (\ref{gw}) proves Proposition \ref{prop:regglue} when the half edges have the same orientation. We conclude the proof by remembering that changing the orientation affects the disk invariants by a factor of $(-1)^{d+1}$.
\end{proof}

\subsection{Orbifold Gluing}\label{sec:orbglue}
\label{sec:orbgl}

Gluing orbifold disks has another level of complexity arising from the twisting at the ramification points of the multiple covers.  At present, we simplify the scenario and show that we can glue disk contributions of $[\C^3/\Z_2]$ to obtain multiple cover contributions of $[\mathcal{O}(-1)\oplus\mathcal{O}(-1)/\Z_2]$ when we use  the weights in Figure \ref{orbwts}.

\begin{figure}
	\centering
		\includegraphics[height=2.2cm]{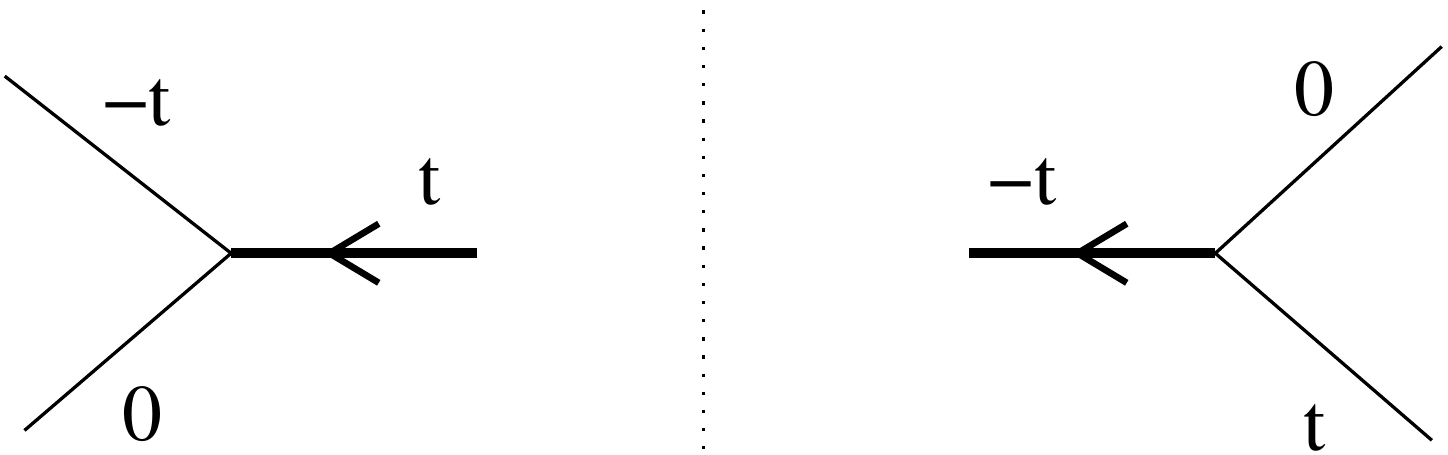}
	\caption{Special weights we use to check the gluing of orbifold disk invariants.}
	\label{orbwts}
\end{figure}

\begin{proposition}
\label{prop:orbg}
Orbifold GW invariants of $[\mathcal{O}(-1)\oplus\mathcal{O}(-1)/\Z_2]$ are obtained by contracting open invariants with same winding at each vertex and scaling the  orbifold Poincare' pairing by a factor of $(-1)^dd$.
\end{proposition}
\begin{remark}
The last part of proposition \ref{prop:orbg} means that open invariants with a twisted (resp. untwisted) origin on one side are multiplied with invariants with a twisted (resp. untwisted) origin on the other side of the edge, the product is scaled by $(-1)^d2d$. Then the two products are added to obtain the total contribution. 
\end{remark}
\begin{proof}
In section \ref{sec:openup} we computed disk invariants for the left vertex.  The right vertex with the given weights and orientation gives the same invariants multiplied by a factor of $(-1)^d$.  In order to glue two \textit{orbifold} disk invariants, we need to have matching windings and \textit{inverse} twisting at the ramification points.  For the $\Z_2$ case, this  means that both origins are twisted or untwisted.  The zero weight at each vertex reduces the circumstance to two cases, either the origins of the disks are marked and twisted, or they are both unmarked (hence, untwisted).  If we multiply two winding $d$ disk invariants twisted at the origin we get:
\begin{equation}
\label{tw}
(-1)^d\left(\frac{1}{2d}\frac{(2d-1)!!}{(2d)!!}\right)^2.
\end{equation}
On the other hand, if we multiply two winding $d$ disk invariants which are untwisted at the origin we get
\begin{equation}
\label{untw}
(-1)^d\left(\frac{1}{2d^2}\right)^2.
\end{equation}
We compare (\ref{tw}) and (\ref{untw}) to the contribution of $d:1$ covers of the twisted $\mathbb{P}^1$ in the orbifold $[\mathcal{O}(-1)\oplus\mathcal{O}(-1)/\Z_2]$.

First consider a $d:1$ cover fully ramified over $0$ and $\infty$ with twisted marks at the ramification points.  Since $f$ maps into  $\mathbb{P}^1\times\mathcal{B}\Z_2$, it classifies a double cover of the source curve fully ramified over the twisted marked points.  Pulling back the tangent bundle to this double cover and only considering the weights of $\Z_2$ invariant sections, we can compute the contribution as
\begin{equation}
\label{twcl}
\frac{1}{2d}\frac{eH^1(\mathcal{O}(-2d)\oplus\mathcal{O}(-2d))}{eH^0(\mathcal{O}(2d))}=\frac{1}{2d}\left(\frac{(2d-1)!!}{(2d)!!}\right)^2.
\end{equation}
where the $2d$ in the denominator corresponds to the global automorphisms of the covers.  Now consider a $d:1$ cover fully ramified over $0$ and $\infty$ with no marked points.  Such a map classifies a double cover of the source curve with no ramification (i.e. two disjoint copies of the source curve).  If we pull back the tangent bundle to the cover, then the $\Z_2$ invariant weights are the weights  for one of the disjoint copies.  Taking into account global automorphisms and the infinitesimal automorphisms at the ramified points of the source curve, one computes the contribution to be
\begin{equation}
\label{untwcl}
\frac{1}{2d}\frac{eH^1(\mathcal{O}(-d)\oplus\mathcal{O}(-d))}{eH^0(\mathcal{O}(2d))}\frac{t}{d}\frac{-t}{d}=\frac{1}{2d^3}.
\end{equation}
The proof is concluded by comparing (\ref{tw}) with (\ref{twcl}) and (\ref{untw}) with (\ref{untwcl}).
\end{proof}

\section{The Closed Crepant Resolution Conjecture via Gluing}\label{sec:closedres}

In this section we deduce the Bryan-Graber crepant resolution conjecture for the orbifold $\mathfrak{X}=[\mathcal{O}(-1)\oplus\mathcal{O}(-1)/\Z_2]$ and its crepant resolution $Y=\mathcal{K}_{\mathbb{P}^1\times\mathbb{P}^1}$ from the results of the previous sections.

We saw in section \ref{sec:orbglue} that there is symmetry in computing open invariants at the two vertices of $[\mathcal{O}(-1)\oplus\mathcal{O}(-1)/\Z_2]$ with the given $\C^*$ action.  In other words, the open potential for the right vertex in Figure \ref{orbwts} can be obtained from the open potential of the left vertex under the change of variables

\begin{align*}
z&\rightarrow \tilde{z}\\
w_d&\rightarrow -\tilde{w_d}\\
\end{align*}

\begin{remark}Throughout the rest of this section, variables with a tilde correspond to formal variables on the right sides of the diagrams.
\end{remark}

Refer to Figure \ref{fig:resol} for the resolution. Computing disk invariants for the right half of the diagram 
with the given orientations and weights leads to the exact same disk invariants computed in section \ref{sec:openup}.  Therefore, the open potential on the right can be obtained from the open potential on the left by the change of variables

\begin{figure}[tb]
	\centering
		\includegraphics[height=3.5cm]{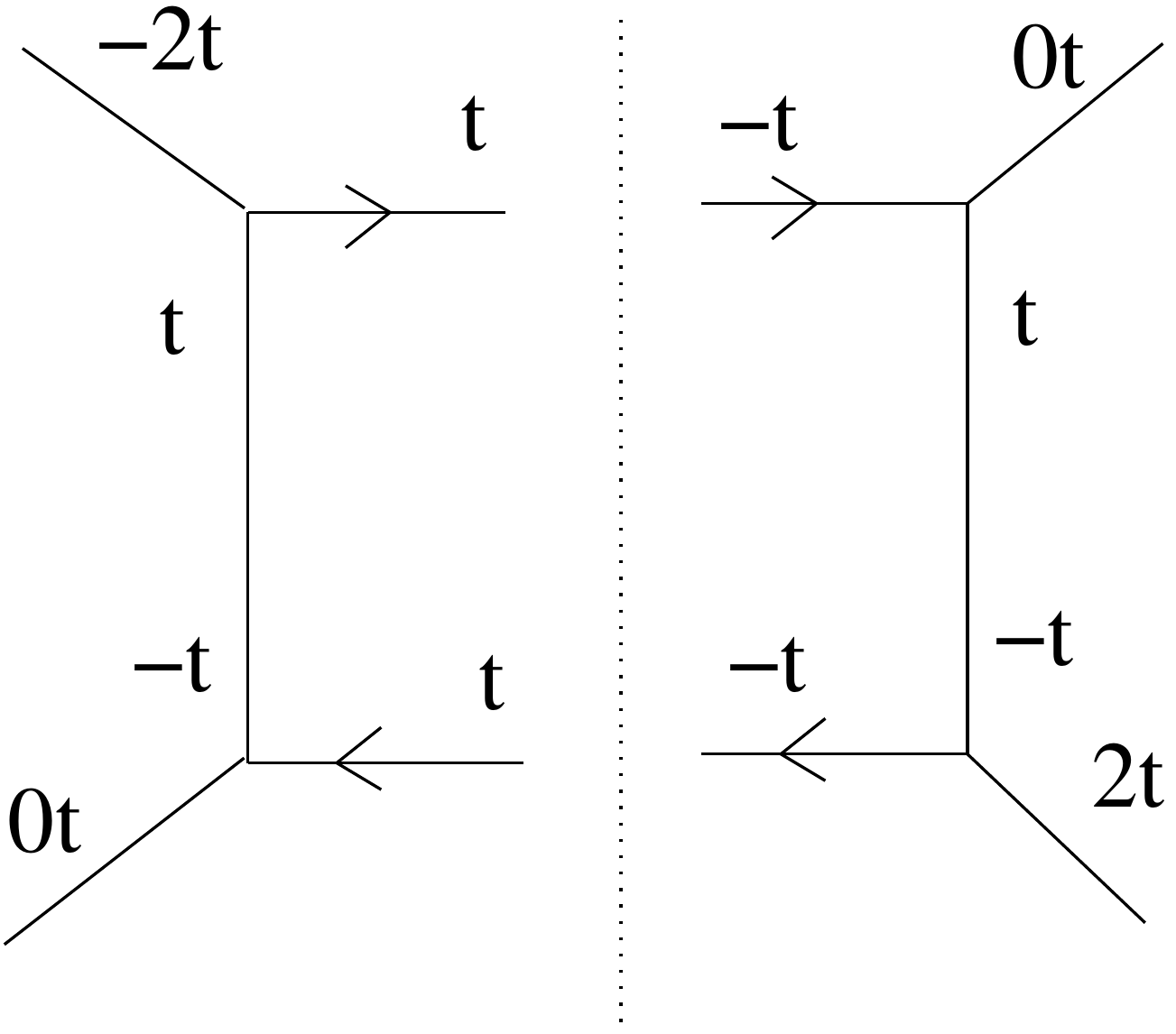}
	\caption{Symmetry in the open potentials of the two open sets of $K_{\proj^1\times\proj^1}$}
	\label{fig:resol}
\end{figure}

\begin{align}
q&\rightarrow \tilde{q} \nonumber \\
x&\rightarrow \tilde{x} \nonumber \\
y_d^{(b)}&\rightarrow \tilde{y}_d^{(t)} \nonumber \\
y_d^{(t)}&\rightarrow \tilde{y}_d^{(b)}
\end{align}

The setup for the crepant resolution conjecture is as follows.
The Chen-Ruan orbifold cohomology of $[\mathcal{O}(-1)\oplus\mathcal{O}(-1)/\Z_2]$ has two generators in degree $2$, the fiber over a point of $\mathbb{P}^1$ and the class of the twisted $\mathbb{P}^1$.  We  assign the formal variables $W$ and $Z$ to correspond to insertions of these classes, respectively.  Any map into the orbifold is classified by the degree on the twisted $\mathbb{P}^1$, thus we only need one degree variable $P$.  On the resolution, we have two insertion variables, corresponding to the fiber over a point in each $\mathbb{P}^1$, let these be $X$ and $Y$.  We also have two degree variables corresponding to the degree of a map on each $\mathbb{P}^1$; denote them  $Q$ and $U$, where $Q$ corresponds to the $\mathbb{P}^1$ which is dual to the divisor corresponding to $X$. 

\begin{theorem}
\label{thm:clcrc}
Under the change of variables
\begin{align}
&Q\rightarrow -1 \nonumber\\
&U\rightarrow -P\nonumber \\
&X\rightarrow iZ \nonumber \\
&Y\rightarrow iZ+W
\end{align}

the genus 0 GW potential of $\mathcal{K}_{\mathbb{P}^1\times\mathbb{P}^1}$ transforms to the genus 0 GW potential of $[\mathcal{O}(-1)\oplus\mathcal{O}(-1)/\Z_2]$ up to stable terms.
\end{theorem}
First we express the two potentials as a sum over the same set of
 decorated trees.  We  then describe how one can extract the contribution to the GW potential from each  tree by multiplying vertex and edge contributions. The open crepant resolution statement proved in section \ref{sec:openres} verifies that the change of variables equates the vertex contributions and edge contributions.

Since the portion of the computation corresponding to degree $0$ maps into the orbifold is immediate from the closed computation done in section \ref{sec:openres}, we focus on contributions with nonzero powers of $U$ and $P$.

\subsection{Closed Invariants of $[\mathcal{O}(-1)\oplus\mathcal{O}(-1)/\Z_2]$}
\label{sec:clinvorb}

The closed potential of the orbifold can be expressed as a sum over localization trees:
\begin{itemize}
\item black (white) vertices of the tree correspond to components contracting to  the left (right) orbifold vertex;
\item edges of the tree correspond to multiple covers of the twisted $\mathbb{P}^1$ obtained by gluing disks.  Each edge is decorated with a positive integer denoting the degree of the multiple cover.
\end{itemize}

By the gluing results of section \ref{sec:orbglue}, closed GW invariants of the orbifold are obtained by gluing open invariants along half edges. For a given localization tree $T$ with more than one edge,   the corresponding contribution to the GW potential is  given by

\begin{equation}
GW_{\mathfrak{X}}(T)=\prod_{\text{black vertices}}\hspace{-.5cm}V(v)\hspace{.1cm}\prod_{\text{edges }e}E(e)\prod_{\text{white vertices}}\hspace{-.5cm}\tilde{V}(v)
\end{equation}

In the above formula, $V(v)$ and $\tilde{V}(v)$ are the open invariants with winding profile corresponding to the edges meeting at $v$ (with the formal variables $z$ and $\tilde{z}$  replaced with $Z$).  In the case that $v$ is univalent, only the contribution from disks with twisted origin is taken. 
The edge contribution is:
\begin{equation}
E(e)=\frac{(-1)^d2d(Pe^W)^d}{w_d\tilde{w}_d}.
\end{equation}
where $e$ is an edge marked with $d$.  The $Pe^W$ is from applying the divisor equation to the new divisor class obtained by gluing and the $(-1)^d2d$ is the gluing factor of section \ref{sec:orbglue}.

In the case that $T'$ is the tree with a unique edge labeled $d$, then one must also take into account the contribution from gluing two unmarked disks.  The contribution in this case is
\begin{equation}
GW_{\mathfrak{X}}(T')=V(v_1)E(e)\tilde{V}(v_2)+\frac{1}{2d^3}(Pe^W)^d.
\end{equation}

\subsection{Closed invariants of $\mathcal{K}_{\mathbb{P}^1\times\mathbb{P}^1}$} 
Again, the Gromov Witten potential is expressed as a sum over localization graphs. For each graph, collapsing all ``vertical" edges (i.e. edges corresponding to multiple covers of the vertical fixed fibers) produces essentially a tree as in section \ref{sec:clinvorb}, with the extra decoration of a subset  $S$ of the edges corresponding to edges mapping to the top invariant line. We forget this extra decoration to organize the potential as a sum over the same trees of section \ref{sec:clinvorb}.





By the results in section \ref{sec:gl},  the contribution to the GW potential from all loci corresponding to a given decorated tree $T$ is:

\begin{equation}
GW_{Y}(T)=\sum_{S\subset \{\text{edges}\}}\left(\prod_{\text{black vertices}}\hspace{-.5cm}V^{(S)}(v)\hspace{.1cm}\prod_{\text{edges }e}E'(e)\prod_{\text{white vertices}}\hspace{-.5cm}\tilde{V}^{(S)}(v)\right)
\end{equation}

In the above formula, $V^{(S)}(v)$ and $\tilde{V}^{(S)}(v)$ are the open GW contributions from all fixed loci with winding profile determined by the edges meeting $v$ (we replace the formal variables $q,\tilde{q}$ with $Q$ and $x,\tilde{x}$ with $X$).  If an adjacent edge is in $S$, this corresponds to a disk mapping to the upper Lagrangian and vice versa.  Also
\begin{equation}
E'(e)=
\begin{cases}
\frac{-d(Ue^Y)^d}{y_d^{(t)}\tilde{y}_d^{(t)}} &\text{if } e\in S\\
\frac{-d(Ue^Y)^d}{y_d^{(b)}\tilde{y}_d^{(b)}} &\text{if } e\notin S
\end{cases}
\end{equation}  
where $e$ is an edge labeled with $d$.  The $-d$ is the gluing factor of section \ref{sec:gl} and the $Ue^Y$ comes from applying the divisor equation to the new divisor class created by gluing.

Let $V'(v)$ and $\tilde{V}'(v)$ denote the open contributions corresponding to all fixed loci with winding profile $(d_1,...,d_n)$ given by the edges $(e_1,...,e_n)$ meeting $v$ (summing over all possibilitiees for the disks to map to the top edge or the bottom edge). Undoing \eqref{undo}, we have:

\begin{center}
$V^{(S)}(v)=
\begin{cases}
\frac{y_{d}^{(t)}}{y_{d}^{(t)}+y_{d}^{(b)}(Qe^X)^{d}}\left(V'(v)-\frac{1}{d^2}y_d^{(b)}\right) &v \text{ univalent, }e\in S\\
\frac{y_{d}^{(b)}(Qe^X)^d}{y_{d}^{(t)}+y_{d}^{(b)}(Qe^X)^{d}}\left(V'(v)-\frac{1}{d^2}y_d^{(b)}\right)+\frac{1}{d^2}y_d^{(b)} &v \text{ univalent, }e\notin S\\
\frac{\prod_{e_i\in S}y_{d_i}^{(t)}\prod_{e_i\notin S}y_{d_i}^{(b)}(Qe^X)^{d_i}}{\prod_{i=1}^n\left(y_{d_i}^{(t)}+y_{d_i}^{(b)}(Qe^X)^{d_i}\right)}V'(v) &\text{else}
\end{cases}
$
\end{center}
and
\begin{center}
$\tilde{V}^{(S)}(v)=
\begin{cases}
\frac{\tilde{y}_{d}^{(t)}(Qe^X)^d}{\tilde{y}_{d}^{(t)}+\tilde{y}_{d}^{(b)}(Qe^X)^{d}}\left(V'(v)-\frac{1}{d^2}\tilde{y}_d^{(t)}\right)+\frac{1}{d^2}\tilde{y}_d^{(t)} &v \text{ univalent, }e\in S\\
\frac{\tilde{y}_{d}^{(b)}}{\tilde{y}_{d}^{(t)}+\tilde{y}_{d}^{(b)}(Qe^X)^{d}}\left(V'(v)-\frac{1}{d^2}\tilde{y}_d^{(t)}\right) &v \text{ univalent, }e\notin S\\
\frac{\prod_{e_i\in S}\tilde{y}_{d_i}^{(t)}(Qe^X)^{d_i}\prod_{e_i\notin S}\tilde{y}_{d_i}^{(b)}}{\prod_{i=1}^n\left(\tilde{y}_{d_i}^{(b)}+\tilde{y}_{d_i}^{(t)}(Qe^X)^{d_i}\right)}V'(v) &\text{else}
\end{cases}
$
\end{center}
\begin{remark}
In each of the above formulas for the vertex contributions, the third case is the generic case and the other two are adjusted to take into account the $\Gamma'$ loci of \eqref{OGW1}.
\end{remark}

\subsection{The Crepant Resolution Transformation}

In order to verify the Bryan-Graber crepant resolution conjecture, we  show that after the prescribed change of variables,
\begin{equation}
GW_{Y}(T)\rightarrow GW_{\mathfrak{X}}(T)
\end{equation}
for every decorated tree $T$.

Even though our formulas for the vertex and edge contributions of $GW_{\mathfrak{X}}(T)$ and $GW_{Y}(T)$ involve winding variables, these variables cancel in the product.  
Hence we can make any substituion for the winding variables and it does not affect the overall product.  Motivated by the open crepant resolution transformation, in the above formulas for $GW_{Y}(T)$ we make the substitutions:
\begin{align*}
y_d^{(b)}&\rightarrow\frac{i}{2}w_d &\tilde{y}_d^{(b)}&\rightarrow\frac{i}{2}(e^{iZ})^d\tilde{w}_d\\
y_d^{(t)}&\rightarrow\frac{i}{2}(-e^{iZ})^dw_d &\tilde{y}_d^{(t)}&\rightarrow(-1)^d\frac{i}{2}\tilde{w}_d\\
Q & \rightarrow -1 & U & \rightarrow -P\\
X & \rightarrow iZ & Y & \rightarrow iZ+W
\end{align*}

By Theorem \ref{thm:ocrc}, under this change of variables $V'(v)\rightarrow V(v)$ and $\tilde{V}'(v)\rightarrow\tilde{V}(v)$.  So for any $S\subset\{\text{edges}\}$, we have:
\begin{equation}
\label{v1}
V^{(S)}(v)\rightarrow
\begin{cases}
\frac{1}{2}V(v)-\frac{i}{4d^2}w_d &v\text{ univalent, }e\in S\\
\frac{1}{2}V(v)+\frac{i}{4d^2}w_d &v\text{ univalent, }e\notin S\\
\frac{1}{2^n}V(v) &\text{else}
\end{cases}
\end{equation}
and similarly,
\begin{equation}
\label{v2}
\tilde{V}^{(S)}(v)\rightarrow
\begin{cases}
\frac{1}{2}\tilde{V}(v)+\frac{i}{4d^2}\tilde{w}_d &v\text{ univalent, }e\in S\\
\frac{1}{2}\tilde{V}(v)-\frac{i}{4d^2}\tilde{w}_d &v\text{ univalent, }e\notin S\\
\frac{1}{2^n}\tilde{V}(v) &\text{else}
\end{cases}
\end{equation}

Also, under the change of variables 
\begin{equation}
\label{ed}
E'(e)\rightarrow 2E(e).
\end{equation}  

Given any tree $T$ with more than one edge, the extra terms on the univalent vertices cancel by summing over all contributions $e\in S$ and $e\notin S$.  Therefore, from \eqref{v1},\eqref{v2} and \eqref{ed};
\begin{align}
\label{gentree}
GW_{Y}(T)&= & \nonumber\\
\sum_{S\subset\{\text{edges}\}}\prod\frac{1}{2}V(v)\prod 2E(e) \prod\frac{1}{2}\tilde{V}(v)&=2^{\#\{\text{edges}\}}\frac{\prod V(v)\prod E(e) \prod\tilde{V}(v)}{2^{\#\{\text{edges}\}}}\nonumber\\
&=GW_{\mathfrak{X}}(T).
\end{align}
If $T'$ is the tree with a unique edge labeled $d$: 
\begin{equation}
\label{spectree}
GW_{Y}(T')=V(v_1)E(e)\tilde{V}(v_2)+\frac{1}{2d^3}(Pe^Y)^d=GW_{\mathfrak{X}}(T').
\end{equation}
Equations \eqref{gentree} and \eqref{spectree} establish Theorem \ref{thm:clcrc}.

\bibliographystyle{alpha}
\bibliography{biblio}

\end{document}